\documentclass[11pt]{article}

\usepackage{amsmath}
\usepackage{amsfonts}
\usepackage{amssymb}
\usepackage{centernot}
\usepackage[margin=1in]{geometry}
\usepackage{enumitem}
\usepackage{hhline}
\usepackage{graphicx,fancybox}
\usepackage{tikz}
\usepackage{subfigure}
\usepackage{amsthm}
\usepackage{pgf}
\usepackage[font={small}]{caption}
\usepackage{etoolbox}
\usepackage{array, makecell, cellspace}

\usepackage{soul}

\DeclareMathOperator*{\argmin}{argmin}
\DeclareMathAlphabet\mbc{OMS}{cmsy}{b}{n}

\renewcommand{\arraystretch}{1.2}

\setlist[enumerate]{label={\rm(\roman*)},topsep=3pt}

\newcommand{\ri}{\operatorname{ri}}
\newcommand{\rec}{\operatorname{rec}}
\newcommand{\cl}{\operatorname{cl}}
\newcommand{\inte}{\operatorname{int}}
\newcommand{\gra}{\operatorname{gra}}
\newcommand{\zer}{\operatorname{zer}}
\newcommand{\Fix}{\operatorname{Fix}}
\newcommand{\ran}{\operatorname{ran}}
\newcommand{\dom}{\operatorname{dom}}
\newcommand{\Id}{\operatorname{Id}}
\newcommand{\prox}{\operatorname{prox}}
\newcommand{\subd}{\widehat{\partial}}

\newcommand{\sign}{\operatorname{sign}}

\newcommand{\Tdr}{T_{\operatorname{aDR}}}

\newcommand{\R}{\mathbb{R}}

\def\tto{\rightrightarrows}

\newcommand{\seq}[2]{{(#1)}_{#2=0}^{\infty}}

\usepackage{titlesec}
\usepackage{titling}
\usepackage{anyfontsize}

\pretitle{\vspace{-1cm}\begin{center}\fontsize{16}{20}\bfseries\sffamily}
\pretocmd{\theequation}{\small}{}{}

\makeatletter
\def\th@plain{%
	\thm@notefont{}
	\itshape 
}
\def\th@definition{%
	\thm@notefont{}
	\normalfont 
}

\usepackage[colorlinks=true,linkcolor=blue,citecolor=blue,urlcolor=blue]{hyperref}
\usepackage[capitalise]{cleveref}

\newtheorem{theorem}{Theorem}[section]
\newtheorem{proposition}[theorem]{Proposition}
\newtheorem{corollary}[theorem]{Corollary}
\newtheorem{lemma}[theorem]{Lemma}
\newtheorem{fact}[theorem]{Fact}

\theoremstyle{definition}
\newtheorem{definition}[theorem]{Definition}

\newtheorem{remark}[theorem]{Remark}

\newtheorem{assumption}[theorem]{Assumption}

\Crefname{fact}{Fact}{Facts}
\Crefname{assumption}{Assumption}{Assumptions}
\Crefname{enumi}{}{}

\title{An Adaptive Alternating Direction Method of Multipliers}

\author{Sedi Bartz \thanks{Department of Mathematical Sciences, Kennedy College of Sciences, University of Massachusetts Lowell, \textsc{MA}, USA. E-mail:~\href{mailto:sedi_bartz@uml.edu}{sedi\_bartz@uml.edu}, \href{mailto:hung_phan@uml.edu}{hung\_phan@uml.edu}} \and Rub\'en Campoy \thanks{Department of Statistics and Operational Research, Universitat de Val\`{e}ncia, Valencia, Spain. E-mail:~\href{mailto:ruben.campoy@uv.es}{ruben.campoy@uv.es}} \and Hung Phan $^*$}

\begin{document}

\maketitle

\begin{abstract}
The alternating direction method of multipliers (ADMM) is a powerful splitting algorithm for linearly constrained convex optimization problems. In view of its popularity and applicability, a growing attention is drawn towards the ADMM in nonconvex settings. Recent studies of minimization problems for noncovex functions include various combinations  of assumptions on the objective function including, in particular, a Lipschitz gradient assumption. We consider the case where the objective is the sum of a strongly convex function and a weakly convex function. To this end we present and study an adaptive version of the ADMM which incorporates generalized notions of convexity and penalty parameters adapted to the convexity constants of the functions. We prove convergence of the scheme under natural assumptions. To this end we employ the recent adaptive Douglas--Rachford algorithm by revisiting the well known duality relation between the classical ADMM and the Douglas--Rachford splitting algorithm, generalizing this connection to our setting. 
We illustrate our approach by relating and comparing to alternatives, and by numerical experiments on a signal denoising problem.

\paragraph{\small MSC\,2020:} 47H05 $\cdot$ 47N10 $\cdot$ 47J25 $\cdot$ 49M27 $\cdot$ 65K15
\end{abstract}

\section{Introduction}
By now, the \emph{alternating direction method of multipliers (ADMM)} is a well-studied and applied splitting algorithm. In particular, it is applied to the problem 
\begin{equation}\label{eq:P}\tag{$\mathcal{P}$}
\begin{array}{rl}
\min & f(x)+g(z)\\[1ex]
\text{s.t.} & Mx=z,\\
& {x\in\R^n,\,z\in\R^m};
\end{array}
\end{equation}
where {$f:\R^n\to{]-\infty,+\infty]}$} and {$g:\R^m\to{]-\infty,+\infty]}$} are proper, lower semicontinuous and convex functions, and $M\in\R^{m\times n}$. The ADMM can be traced back to 1975 in the studies of Glowinski and Marroco \cite{GM75}, and of Gabay and Mercier \cite{GM76}. It was revisited in the early 1980s in~\cite{FG83,Gab83}. The ADMM has been successfully applied to a wide range of statistical and  learning problems such as sparse regression, signal and image processing, and support vector machines, to name a few. An extensive survey on the ADMM and its applications can be found in \cite{BPCPE11}.

The ADMM can be viewed as an enhanced version of the \emph{method of multipliers} in the case where the objective function is separable. The \emph{augmented Lagrangian} associated with \eqref{eq:P}  is the function
\begin{equation}\label{e:lagrangian}
L_{\gamma}(x,z,y)=f(x)+g(z)+\langle y, Mx-z\rangle+\frac{\gamma}{2}\|Mx-z\|^2,
\end{equation}
where $\gamma\geq 0$ is the \emph{penalty parameter} and $y\in\R^m$ is the \emph{Lagrange multiplier}. By employing the method of multipliers, one solves \eqref{eq:P} by iteratively minimizing $L_{\gamma}(x,z,y)$ over the (primal) variables $x$ and $z$ while updating the Lagrange multiplier $y$ (the dual variable). However, this requires to minimize the Lagrangian jointly in $x$ and $z$. In order to avoid this situation, the ADMM takes advantage of the separability of the objective function and splits the minimization procedure into two separate steps, one for each variable. Specifically, by fixing a positive penalty parameter $\gamma$, the iterative step  of the  ADDM  for solving \eqref{eq:P} is
\begin{subequations}\label{alg0}
\begin{align}
x^{k+1}&=\argmin_{x\in \R^n}
L_\gamma(x,z^k,y^k),
\label{alg0:x}\\
z^{k+1}&=
\argmin_{z\in \R^m}
L_\gamma(x^{k+1},z,y^k),
\label{alg0:z}\\
y^{k+1}&=y^k+\gamma(Mx^{k+1}-z^{k+1}).
\end{align}
\end{subequations}
Convergence of this scheme is well established in the case where $f$ and $g$ are convex, see, e.g., \cite[\S~3.2]{BPCPE11}. In nonconvex cases it has been studied, e.g., \cite{HLR16,LP15,WCX18,WYZ19}, under various combinations of assumptions which include, in particular, a Lipschitz continuity assumption on the gradient of $f$ and/or $g$.

In the present study we consider the case where $f$ is strongly convex and $g$ is weakly convex such that the objective of~\eqref{eq:P} is convex on the constraint. We introduce an \emph{adaptive alternating direction method of multipliers (aADMM)} for which we incorporate a flexible range of penalty parameters adapted to the convexity constants of the functions $f$ and $g$.   To this end we revisit the well-known relation between the classical ADMM and the \emph{Douglas--Rachford} (DR) splitting algorithm \cite{DR56,LM79}. This duality relation was first observed in~\cite[\S~5.1]{Gab83} and later revisited by other authors, see, e.g.,~\cite[Appendix~A]{ACL16} or \cite[Remark~3.14]{BK15}. A more detailed discussion regarding the ADMM is available in \cite{EY15} while \cite{MZ19} is a recent survey on equivalences and other relations between splitting algorithms. We provide an analogous relation between our aADMM and the recent \emph{adaptive Douglas--Rachford (aDR)} algorithm \cite{BDP19,DPadapt}. We then employ this relation in order to derive convergence of our aADMM from the convergence of the aDR.

We point out (see Remark~\ref{rem:reformulation}) that in our strongly-weakly convex setting, the functions $f$ and $g$ in problem~\eqref{eq:P} can be augmented into convex functions which transform the problem into a convex one, admissible for the classical ADMM, with the same minimizers, optimal values and computational difficulty level. However, the ADMM for the augmented problem corresponds to a {Douglas--Rachford} algorithm which is not in direct duality relations with the original strongly-weakly problem. Consequently, this theoretical aspect is lacking. Instead, we preserve and analyze the original problem. Our approach does yield a natural duality relation with a corresponding aDR algorithm which is instrumental in our convergence analysis. An additional benefit of our approach is that we relax and improve previously imposed assumptions on the strongly-weakly convex scenario such as in \cite{MSMC15,wADMM} (see Remark~\ref{rm:wsc1} and Remark~\ref{rm:wsc2}). Finally, although augmentation is a viable option in the strongly-weakly convex setting, application of the adaptive algorithms to the original problem has its own merit and by now was studied in a number of recent publications such as \cite{BCP20,BDP19,DPadapt,DP3op,GW19,GHY17,LR21,wADMM,ZHF20}, to name a few.

Our main result is \Cref{t:cvg_aadmm}, where we provide convergence of the aADMM. Moreover, in order to show how our framework generalizes and relaxes the framework of the classical convex ADMM, we incorporate in our convergence analysis the relaxed assumptions regarding the convexity of $f$ and $g$ while not imposing further on top of the traditional constraint qualifications of the ADMM. To this end we revisit and incorporate in our analysis some of the most commonly imposed assumptions and conditions on the classical ADMM. For the sake of accessibility and convenience, we summarize and unify our analysis with these classical conditions into an integrated tool in \Cref{cor:aADMM} for the aADMM and, in particular, in \Cref{cor:ADMM} for the classical ADMM.

 Finally, we illustrate computational aspects of both approaches (our aADMM and the classical ADMM on an equivalent modified problem) by numerical experiments on a signal denoising problem with a weakly convex regularization term.

The remainder of the paper is organized as follows. In \Cref{sec:prelim} we recall basic definitions and preliminary results. In \Cref{sec:aDR} we recall notions of generalized monotonicity and the convergence of the adaptive Douglas--Rachford algorithm. In \Cref{sec:aADMM} we introduce our adaptive ADMM, we analyze some of its basic properties and we discuss conditions and constraint qualifications. The convergence of the scheme is established in \Cref{sec:conv}. In \Cref{sec:exp} we conduct numerical experiments on a signal denoising problem with a weakly convex regularization term. Finally, we conclude our discussion in \Cref{sec:con}.

\section{Preliminaries}\label{sec:prelim}
Throughout, $\langle \cdot,\cdot\rangle$ denotes the inner product in $\R^n$ with induced norm $\|\cdot\|$ defined by $\|x\|=\sqrt{\langle x,x\rangle},\ x\in\R^n$. We set $\R_+:=\{r\in\R: r\geq 0\}$ and $\R_{++}:=\{r\in\R: r>0\}$. Let $M\in\R^{m\times n}$. Then $\ran M$, $\ker M$ and $\|M\|$ denote, respectively, the \emph{range}, the \emph{null space} and the matrix $2$-norm of $M$. Let $C\subseteq R^n$ be a set. The \emph{closure}, the \emph{interior} and the \emph{relative interior} of $C$ are denoted by $\cl C$, $\inte C$ and $\ri C$, respectively. We denote by $A:\R^n\tto\R^n$ a \emph{set-valued operator} that maps any point $x\in\R^n$ to a set $A(x)\subseteq\R^n$. In the case where $A$ is single-valued, we write  $A:\R^n\to\R^n$. The \emph{graph}, the \emph{domain}, the \emph{range}, the set of \emph{fixed points} and the set of \emph{zeros} of A, are denoted, respectively, by $\gra A$, $\dom A$, $\ran A$, $\Fix A$ and $\zer A$, i.e.,
\begin{gather*}
\gra A:=\left\{(x,u)\in\R^n\times\R^n : u\in A(x)\right\},\quad \dom A:=\left\{x\in\R^n : A(x)\neq\emptyset\right\},\\
\ran A:=\left\{x\in\R^n : x\in A(z) \text{ for some } z\in\R^n  \right\},\\
\Fix A:=\left\{x\in\R^n : x\in A(x)\right\}\quad
\text{and} \quad \zer A:=\left\{x\in\R^n : 0\in A(x)\right\}.
\end{gather*}
The \emph{inverse} of $A$, denoted by $A^{-1}$, is the operator defined via its graph by $\gra A^{-1}:=\{(u,x)\in \R^n \times \R^n: u\in A(x)\}$. We denote the \emph{identity} mapping by $\Id$. 
The \emph{resolvent} of the operator $A:\R^n\tto\R^n$ with parameter $\gamma>0$ is the operator $J_{\gamma A}$ defined by
\begin{equation*}
J_{\gamma A}:=(\Id+\gamma A)^{-1}.
\end{equation*}
The \emph{$\lambda$-relaxed resolvent} of $A$ with parameter $\gamma>0$ is the operator $J^{\lambda}_{\gamma A}$ defined by
$$J^{\lambda}_{\gamma A}:=(1-\lambda)\Id + \lambda J_{\gamma A}.$$%

\begin{definition}
Let $D\subseteq\R^n$ be a nonempty set. The mapping $T:D\to\R^n$ is said to be
\begin{enumerate}
\item \emph{Lipschitz continuous} with \emph{Lipschitz constant} $l>0$ if
\begin{equation*}
\|T(x)-T(y)\|\leq l\|x-y\|, \quad \forall x,y\in D;
\end{equation*}
\item \emph{nonexpansive} if it is Lipschitz continuous with constant $l=1$;
\item \emph{conically $\theta$-averaged}, where $\theta>0$, if there exists a nonexpansive mapping $R:D\to\R^n$ such that
\begin{equation*}
T=(1-\theta)I+\theta R.
\end{equation*}
\end{enumerate}
\end{definition}
Conically $\theta$-averaged mappings were studied in~\cite{BMW20}, in which they were referred to as \emph{conically nonexpansive mappings}. They can be viewed as a natural extension of the classical $\theta$-averaged mappings (see, e.g.,~\cite[Definition~4.33]{BC17}). Additional properties and detailed discussions can be found in~\cite{BDP19,GW19}.

An extended real-valued function $f:\R^n\to{]-\infty,+\infty]}$ is said to be \emph{proper} if its (effective) domain, the set $\dom f:=\{x\in\R^n : f(x)<+\infty\}$, is nonempty. We say that $f$ is \emph{lower semicontinuous (l.s.c.)} if, at any $\bar{x}\in\R^n$,
$$f(\bar x)\leq \liminf_{x\to\bar x} f(x).$$
Let $\alpha\in\R$. We say that $f$ is $\alpha$-convex if $f-\frac{\alpha}{2}\|\cdot\|^2$ is convex, equivalently, if
\begin{equation*}
f((1-\lambda)x+\lambda y)\leq \lambda f(x)+(1-\lambda)f(y)-\frac{\alpha}{2}\lambda(1-\lambda)\|x-y\|^2, \quad\forall x,y\in\R^n,\ \lambda\in[0,1].
\end{equation*}
In particular, $f$ is convex if and only if $f$ is $0$-convex. For an $\alpha$-convex function $f$, we say that $f$ is \emph{strongly convex} if $\alpha>0$ and we say that $f$ is \emph{weakly convex} (or \emph{hypoconvex}) if $\alpha<0$. It follows that if $f_1$ is $\alpha_1$-convex and $f_2$ is  $\alpha_2$-convex, then $f_1+f_2$ is $(\alpha_1+\alpha_2)$-convex. 

The function $f$ is {\em coercive} if
\begin{equation*}
\lim_{\|x\|\to\infty}f(x)=+\infty
\end{equation*}
and {\em supercoercive} if
\begin{equation*}
\lim_{\|x\|\to\infty}\frac{f(x)}{\|x\|}=+\infty.
\end{equation*}
One can verify that (see, e.g., \cite[Corollary~11.17]{BC17})
\begin{equation*}
\text{strong convexity }\implies
\text{ supercoercivity }\implies
\text{coercivity}.
\end{equation*}

Let $\gamma>0$. The \emph{proximal operator} with parameter $\gamma$ associated with the function $f$ is defined by
\begin{equation*}
\prox_{\gamma f}: \R^n\tto\R^n:y\mapsto\argmin_{x\in\R^n}\left\{ f(x)+\frac{1}{2\gamma}\|x-y\|^2\right\}, \quad \forall y\in\R^n.
\end{equation*}
Let $x\in\dom f$. The \emph{(convex) subdifferential} of $f$ at $x$ is the set 
\begin{equation*}
\partial f(x):=\{u\in\R^n : \langle y-x,u\rangle +f(x)\leq f(y),\quad \forall y\in\R^n\}.
\end{equation*}
The \emph{Fr\'echet subdifferential} of $f$ at $x$ is the set 
\begin{equation*}
\subd f(x):=\left\{u\in\R^n : \liminf_{\substack{y\to x\\ y\neq x}} \frac{f(y)-f(x)-\langle u,y-x\rangle}{\|y-x\|}\geq 0\right\}.
\end{equation*}
When $f$ is differentiable at $x$, we denote its gradient at $x$ by $\nabla f(x)$. We recall the following facts regarding subdifferentials and gradients.

\begin{fact}\label{f:subdif}
Let $f,g:\R^n\to{]-\infty,+\infty]}$ be proper and let $M\in\R^{n\times m}$.  
\begin{enumerate}
\item $\partial f(x)\subseteq \subd f(x)$ for all $x\in\dom f$.\label{f:subdif_inclusion}
\item If $f$ is convex, then $\subd f(x)=\partial f(x)$ for all $x\in\dom f$.\label{f:subdif_convex}
\item If $f$ is differentiable at $x\in\dom f$, then $\subd f(x)=\{\nabla f(x)\}$.\label{f:subdif_dif}
\item $0\in\partial f(\bar{x})$ if and only if $\bar{x}\in\dom f$ minimizes $f$ over $\R^n$.\label{f:subdif_minc}
\item $0\in \subd f(\bar{x})$ if $f$ attains a local minimum at $\bar{x}\in\dom f$.\label{f:subdif_min}
\item If $g$ is differentiable at $x\in\dom f\cap \dom g$, then $\subd(f+g)(x)=\subd f(x)+\nabla g(x)$.\label{f:subdif_sumdif}
\item\label{f:subdif_sum}
If $f$ and $g$ are lower semicontinuous, then 
$$\subd f(x)+\subd g(x)\subseteq \subd(f+g)(x) \ \ \text{for all}\ x\in\dom f\cap\dom g.$$
\item\label{f:subdif_chain}
If $f$ is lower semicontinuous, then 
$$M^T\subd f(Mx)\subseteq \subd(f\circ M)(x)\ \ \text{for all}\ x\in\R^m\ \text{such that}\ Mx\in\dom f.$$
\end{enumerate}
\end{fact}
\begin{proof}
\Cref{f:subdif_inclusion}: See, e.g.,~\cite[Proposition 8.6]{RW}.
\Cref{f:subdif_convex}: See, e.g.,~\cite[Proposition 8.12]{RW}.
\Cref{f:subdif_dif}: See, e.g.,~\cite[Proposition 1.1]{K03}.
\Cref{f:subdif_minc}: See, e.g.,~\cite[Theorem~16.3]{BC17}.
\Cref{f:subdif_min}: See, e.g.,~\cite[Proposition 1.10]{K03}.
\Cref{f:subdif_sumdif}: See, e.g.,~\cite[Corollary 1.12.2]{K03}.
\Cref{f:subdif_sum}: See, e.g.~\cite[Corollary 10.9]{RW}.
\Cref{f:subdif_chain}: See, e.g,~\cite[Theorem~10.6]{RW}.%
\end{proof}

In order to carry out our analysis we adopt the Fr\'echet subdifferential. However, our approach is applicable if one adopts other notions of the subdifferentials, such as the \emph{Mordukhovich} subdifferential or the \emph{Clarke-Rockafellar} subdifferential since these notions coincide for $\alpha$-convex functions (see~\cite[Proposition~6.3]{BMW20}).

Finally, let $f,g:\R^n\to {]}-\infty,+\infty]$ be proper and convex. We recall that the \emph{recession function} of $f$ is defined by
\begin{equation*}
\rec f:\R^n\to {]}-\infty,+\infty]:
y \mapsto \sup_{x\in\dom f}\{f(x+y) - f(x)\},
\end{equation*}
the \emph{Fenchel conjugate} of $f$ is defined by
\begin{equation*}
f^*:\R^n\to {]}-\infty,+\infty]:
u \mapsto \sup_{x\in\R^n}\{\langle u,x\rangle - f(x)\},
\end{equation*}
and the infimal convolution of $f$ and $g$ is defined by
\begin{equation*}
f\square g:\R^n\to {[}-\infty,+\infty]:
x\mapsto \inf_{y\in\R^n}\{f(y)+g(x-y)\}.
\end{equation*}

\section{Generalized Monotonicity and the Adaptive Douglas-Rachford Algorithm}\label{sec:aDR}
We recall the following notions of generalized monotonicity.
\begin{definition}\label{def:monotone}
Let $A:\R^n\tto\R^n$ and let $\alpha\in\R$. Then $A$ is said to be
\begin{enumerate}
\item \emph{$\alpha$-monotone} if \label{def:monotone_mono}
\begin{equation*}
\langle x-y,u-v\rangle\geq \alpha\|x-y\|^2,\quad\forall (x,u),(y,v)\in\gra A;
\end{equation*}
\item \emph{$\alpha$-comonotone} if $A^{-1}$ is $\alpha$-monotone, i.e., \label{def:monotone_comono}
\begin{equation*}
\langle x-y,u-v\rangle\geq \alpha\|u-v\|^2,\quad\forall (x,u),(y,v)\in\gra A.
\end{equation*}
\end{enumerate}
An $\alpha$-monotone (resp. $\alpha$-comonotone) operator $A$ is said to be \emph{maximally $\alpha$-monotone} (resp. \emph{maximally $\alpha$-comonotone}) if there is no $\alpha$-monotone (resp. $\alpha$-comonotone) operator $B:\R^n\tto\R^n$ such that $\gra A$ is properly contained in $\gra B$.
\end{definition}

\begin{remark}
We note that in the case where $\alpha=0$ in \Cref{def:monotone}, both $0$-monotonicity and $0$-comonotonicity simply mean \emph{monotonicity} (see, for example, \cite[Definition~20.1]{BC17}). In the case where $\alpha>0$, $\alpha$-monotonicity is also referred to as \emph{strong monotonicity} (see, e.g.,~\cite[Definition~22.1(iv)]{BC17}) and $\alpha$-comonotonicity is also referred to as \emph{cocoercivity} (see, e.g.,~\cite[Definition~4.10(iv)]{BC17}). In the case where $\alpha<0$, $\alpha$-monotonicity is also referred to as \emph{hypomonotonicity} (or \emph{weak monotonicity}) and $\alpha$-comonotonicity is also referred to as \emph{cohypomonotonicity} (see, e.g,~\cite[Definition 2.2]{CP04}).
\end{remark}

\begin{fact}[maximal monotonicity of the subdifferential]\label{f:subd_maxmon}
Let $\alpha\in\R$ and suppose that ${f:\R^n\to\left]-\infty,+\infty\right]}$ is $\alpha$-convex. Then the Fr\'echet subdifferential of~$f$, $\widehat{\partial} f:\R^n\tto\R^n$, is maximally $\alpha$-monotone.%
\end{fact}
\begin{proof}
See, e.g.,~\cite[Lemma~5.2]{DPadapt}.%
\end{proof}

We recall that under certain assumptions on the monotonicity parameters, the resolvents of comonotone operators are conically averaged.

\begin{fact}[resolvents of comonotone operators]\label{f:resolvent_comono}
	Let $\alpha\in\R$ and set $\gamma>0$ such that $\gamma>-\alpha$. If $A:\R^n\tto\R^n$ is $\alpha$-comonotone, then
\begin{enumerate}[label={\rm(\roman*)}]
\item $J_{\gamma A}$ is single-valued and conically $\frac{\gamma}{2(\gamma+\alpha)}$-averaged; \label{f:resolvent_comono_I}
\item $\dom J_{\gamma A}=\R^n$ if and only if $A$ is maximally $\alpha$-comonotone. \label{f:resolvent_comono_II}
\end{enumerate}
\end{fact}
\begin{proof}
See~\cite[Propositions~3.7 and 3.8(i)]{BDP19} and \cite[Proposition~3.7(v)\&(vi)]{BMW20}.%
\end{proof}

\begin{fact}[maximal comonotonicity]\label{f:maxcom}
Let $\alpha\in\R$ and $A:\R^n\tto\R^n$. The following hold.
\begin{enumerate}
\item $A$ is maximally $\alpha$-comonotone $\iff$ $A^{-1}$ is maximally $\alpha$-monotone.\label{f:maxcom_I}
\item Suppose that $\alpha\geq 0$. Then \label{f:maxcom_II}
\begin{equation*}
A \text{ is maximally } \alpha\text{-comonotone}\iff \begin{array}{c} A \text{ is } \alpha\text{-comonotone}\\ \text{and maximally monotone.}\end{array}
\end{equation*}
\end{enumerate}
\end{fact}
\begin{proof}
\ref{f:maxcom_I}: Follows directly from Definition~\ref{def:monotone}. \ref{f:maxcom_II}: Apply \cite[Proposition~3.5(i)]{DPadapt} to the operator $A^{-1}$ (alternatively, see~\cite[Proposition~3.2(ii)]{BDP19}).%
\end{proof}

\begin{lemma}[closedness of graph]\label{l:graph} 
Let $\alpha\in\R$ and let $A:\R^n\tto\R^n$ be maximally $\alpha$-comonotone. Then $\gra A$ is closed.
\end{lemma}
\begin{proof}
Let $\seq{x_k,u_k}{k}\subseteq\gra A$ such that $(x_k,u_k)\to (x,u)\in\R^n\times\R^n$. By employing \Cref{f:maxcom}\ref{f:maxcom_I}, we see that $A^{-1}$ is maximally $\alpha$-monotone, which is equivalent to $B:=A^{-1}-\alpha\Id$ being maximally monotone. Consequently, $\gra B$ is closed (see, for example, \cite[Proposition 20.38]{BC17}). Since $\seq{u_k,x_k-\alpha u_k}{k}\subseteq \gra B$ and $(u_k,x_k-\alpha u_k)\to (u,x-\alpha u)$, we conclude that $(u,x-\alpha u)\in\gra B$, which, in turn, implies that $(x,u)\in \gra A$.%
\end{proof}

We conclude this section by recalling the convergence of the \emph{adaptive Douglas--Rachford (aDR)} algorithm for maximally comonotone operators. The aDR can be viewed as an extension of the classical Douglas--Rachford splitting algorithm~\cite{DR56,LM79}, originally utilized to find a zero of the sum of two maximally monotone operators by employing their resolvents. The aDR algorithm was recently presented and studied in \cite{DPadapt} in order to find a zero of the sum of a strongly monotone operator and a weakly monotone operator. This analysis was later extended in \cite{BDP19} to include, in particular, the case of a strongly comonotone operator and a weakly comonotone operator. Convergence results for the shadow sequence of the aDR (i.e., the image of the aDR sequence under the resolvent) in infinite-dimensional spaces have been recently provided in~\cite{BCP20}. We recall the following fact regarding the convergence of the aDR for comonotone operators.

\begin{fact}[aDR for comonotone operators]\label{f:aDR_comono}
Let $\alpha,\beta\in\R$ be such that $\alpha+\beta\geq 0$. Let $A:\R^n\tto\R^n$ be a maximally $\alpha$-comonotone operator and let $B:\R^n\tto\R^n$ be a maximally $\beta$-comonotone operator such that $\zer(A+B)\neq \emptyset$. Suppose that $(\gamma,\delta)\in\R^2_{++}$ satisfy
\begingroup
\allowdisplaybreaks
\begin{subequations}\label{e:aDR_comono}
\begin{align}
0<\gamma+2\alpha=\delta, &\quad \text{ if } \alpha+\beta=0,\label{eq:assumtions_comonoI}\\
\text{or } \quad (\gamma+\delta)^2<4(\gamma+\alpha)(\delta+\beta), &\quad \text{ if } \alpha+\beta>0;\label{eq:assumtions_comonoII}
\end{align}
\end{subequations}
\endgroup
and set $(\lambda,\mu)\in\R^2_{++}$ by
\begin{equation}\label{e:lm_mu}
(\lambda-1)(\mu-1)=1 \quad\text{and}\quad \delta=(\lambda-1)\gamma.
\end{equation}
Set further $\kappa\in{]0,\overline{\kappa}[}$ where
\begin{equation}\label{e:aDR_kappa}
\overline{\kappa}:=\left\{\begin{array}{ll}
1, & \text{if } \alpha+\beta=0;\\
\displaystyle\frac{4(\gamma+\alpha)(\delta+\beta)-(\gamma+\delta)^2}{2(\gamma+\delta)(\alpha+\beta)}, & \text{if } \alpha+\beta>0.
\end{array}\right.
\end{equation}
Finally, set $x_0\in\R^n$ and let $(x_k)_{k=0}^{\infty}$ be generated by the recurrence
\begin{equation}\label{eq:aDR_iter}
x_{k+1}=\Tdr(x_k):=(1-\kappa)x_k+\kappa J^{\lambda}_{\gamma A}J^{\mu}_{\delta B}(x_k),\quad k=0,1,2,\ldots.
\end{equation}
Then 
\begin{enumerate}[label={\rm(\roman*)}]
\item $x_k\to x^\star\in \Fix \Tdr \text{ and } J_{\delta B}(x^\star)\in \zer(A+B);$  \label{f:aDR_comonoI}
\item $J_{\delta B}(x_k)\to J_{\delta B}(x^\star)\in \zer(A+B);$\label{f:aDR_comonoII}
\item $J_{\delta B}(x_k)-J_{\gamma A}J^{\mu}_{\delta B}(x_k)\to 0$.\label{f:aDR_comonoIII}
\end{enumerate}
\end{fact}
\begin{proof}
We note that \eqref{e:aDR_comono} implies that $\alpha+\gamma>0$ and $\delta+\beta>0$. In view of \Cref{f:resolvent_comono}, $J_{\gamma A}$ and $J_{\delta B}$ are single-valued with full domain and, consequently, the iteration in \eqref{eq:aDR_iter} is well defined.
By \cite[Theorem 5.4]{BDP19}, we arrive at
\[
x_k\to x^\star\in \Fix\Tdr \quad\text{and}\quad x_{k}-\Tdr(x_{k})\to 0
\]
which, combined with \cite[Lemma~4.1]{DPadapt}, implies \Cref{f:aDR_comonoI} and \Cref{f:aDR_comonoIII}. Finally, by invoking \Cref{f:resolvent_comono}\Cref{f:resolvent_comono_I} we see that $J_{\delta B}$ is conically averaged. This implies that $J_{\delta B}$ is Lipschitz continuous and, consequently,  \Cref{f:aDR_comonoII} follows~from~\Cref{f:aDR_comonoI}.%
\end{proof}

\section{Adaptive ADMM}\label{sec:aADMM}

The adaptive alternating direction method of multipliers requires natural generalized convexity assumptions as well as traditional assumptions on \eqref{eq:P}. We divide these conditions and constraint qualifications into three categories: generalized convexity assumptions, existence of solutions for \eqref{eq:P} and existence and well posedness of our iterative steps. Compared with the traditional framework of the classical ADMM, we show that our settings are more general and admit a wider class of functions within the first category while maintaining  traditional assumptions and constraint qualifications in the second and third categories. To this end we recollect the most common and widely imposed conditions on the classical ADMM as well as equivalences and relations between them. We {divide} our discussion into the following subsections: convexity qualifications, critical points and minimizes, introduction of the aADMM, constraint qualifications and related conditions for existence of our iterative steps.

\subsection{Convexity Assumptions}

One of the underlying assumptions for the classical ADMM is that the functions $f$ and $g$ in~\eqref{eq:P} are proper, lower semicontinuous and convex. We adapt to a wider class of functions via the following natural assumption.

\begin{assumption}\label{as:conv_param}
Let $M\in\R^{m\times n}$ be a nonzero matrix. We assume that the function $f:\R^n\to\left]-\infty,+\infty\right]$ is proper, lower semicontinuous and $\alpha$-convex, and $g:\R^m\to\left]-\infty,+\infty\right]$ is proper, lower semicontinuous and $\beta$-convex where $\alpha, \beta\in\R$ are parameters such that 
\begin{equation*}
\alpha\geq 0\quad\text{and}\quad\alpha+\beta\|M\|^2\geq 0.
\end{equation*}
\end{assumption}

In order to characterize the solutions of \eqref{eq:P} under \cref{as:conv_param}, we will employ the following lemma.

\begin{lemma}
\label{l:goM}
Let $M\in\R^{m\times n}$. Suppose that $g:\mathbb{R}^m\to\left]-\infty,+\infty\right]$ is a $\beta$-convex function where $\beta<0$. Then $g\circ M$ is $\beta\|M\|^2$-convex.
\end{lemma}
\begin{proof} Let $x, y\in\R^n$ and $\lambda\in [0,1]$. Then the $\beta$-convexity of $g$ implies that
\begin{align*}
(g\circ M)\big((1-\lambda)x+\lambda y\big)=~&g\big((1-\lambda)Mx+\lambda My\big)\\
\leq~& (1-\lambda)g(Mx)+\lambda g(My)\\& -\dfrac{\beta}{2}\lambda(1-\lambda)\|Mx-My\|^2\\
\leq~& (1-\lambda)(g\circ M)(x)+\lambda(g\circ M)(y)\\& -\Big(\dfrac{\beta}{2}\|M\|^2\Big)\lambda(1-\lambda)\|x-y\|^2,
\end{align*}
i.e., $g\circ M$ is $\beta\|M\|^2$-convex.%
\end{proof}

\begin{lemma}\label{rem:conv_sum}
Let \Cref{as:conv_param} hold. Then $(x^\star,z^\star)$ is a solution of \eqref{eq:P} if and only if
\begin{equation*}
Mx^\star=z^\star \quad\text{and}\quad 0\in\partial(f+g\circ M)(x^\star).
\end{equation*}
\end{lemma}
\begin{proof}
We note that \eqref{eq:P} is equivalent to the unconstrained optimization problem of minimizing $f+g\circ M$ over $\R^n$. We claim that $f+g\circ M$ is a convex function. Indeed, under \Cref{as:conv_param}, if $\beta\geq 0$, then $f$ as well as $g$ are convex and so is $f+g\circ M$. If $\beta<0$, then \cref{l:goM} implies that $g\circ M$ is $\beta\|M\|^2$-convex. Consequently, we see that $f+g\circ M$ is $(\alpha+\beta\|M\|^2)$-convex. In particular, $f+g\circ M$ is convex since $\alpha+\beta\|M\|^2\geq 0$. Finally, by recalling \Cref{f:subdif}\Cref{f:subdif_minc}, we conclude that the minimizers of \eqref{eq:P} are characterized by the first order optimality condition $0\in\partial(f+g\circ M)(x^\star)$.%
\end{proof}

\begin{remark}[on strongly-weakly convex settings]
\label{rem:reformulation}
Under \cref{as:conv_param}, problem \eqref{eq:P} can be referred to as a strongly-weakly convex problem, see, e.g., \cite{DPadapt,GHY17}. Splitting methods for this problem require computability of subdifferentials and their resolvents. We observe that by setting
\begin{equation*}
\tilde{f}:=f+\frac{\beta}{2}\|M(\cdot)\|^2\ \quad\ \text{and}\ \quad\ \tilde{g}:=g-\frac{\beta}{2}\|\cdot\|^2,
\end{equation*}
\eqref{eq:P} is equivalent to
\begin{equation}\label{eq:Phat}\tag{$\widehat{\mathcal{P}}$}
\min \big(\tilde{f}(x)+\tilde{g}(z)\big)
\quad\text{s.t.}\quad Mx=z.
\end{equation}
Under \cref{as:conv_param}, both $\tilde{f}$ and $\tilde{g}$ are convex. Indeed, a straightforward verification implies that $\tilde{g}$ is convex. Furthermore, if $\beta\geq 0$, then $\tilde{f}$ is convex because $f$ and $\frac{\beta}{2}\|M(\cdot)\|^2$ are convex. If $\beta<0$, then $\frac{\beta}{2}\|M(\cdot)\|^2$ is $\beta\|M\|^2$-convex by \cref{l:goM}. We note that $f$ is $\alpha$-convex and $\alpha+\beta\|M\|^2\geq 0$, so $\tilde{f}$ is $(\alpha+\beta\|M\|^2)$-convex, in particular, convex.

Consequently, one can apply the classical ADMM to~\eqref{eq:Phat} in order to solve \eqref{eq:P} with a similar computational difficulty level. A similar strategy was pointed out as an alternative to the adaptive DR algorithm in \cite[Remark~4.15]{DPadapt}.
\end{remark}

\begin{remark}[A non-symmetric scenario]
We would like to emphasize that this approach is not symmetric with respect to the weakly-strongly convexity assumptions; that is, we do not allow $f$ to be weakly-convex. The main reason is that even in the case where $g$ is strongly convex, we cannot guarantee the strong convexity of the composition $g\circ M$, as we did with the weak convexity in \Cref{l:goM}. Thus, we do not assess the convexity of the equivalent problem \eqref{eq:Phat} discussed in \Cref{rem:reformulation}.  
\end{remark}

We will experiment with the approach outlined in Remark~\ref{rem:reformulation} in Section~\ref{sec:exp}. However, from the theoretical perspective, we pursue a different path: We do not modify~\eqref{eq:P}, instead, we provide an adaptive version of the ADMM which is admissible under the strongly-weakly convex setting of~\eqref{eq:P}. To this end we provide duality relations with the recent adaptive DR algorithm \cite{DPadapt} which, in turn, are instrumental in the proof of convergence of our adaptive scheme. One of the justifications of our approach is that it complements and extends the natural and well known duality relation between the classical ADMM and the classical DR algorithm to the strongly-weakly convex setting. Experiments with our approach and some comparisons to the approach in Remark~\ref{rem:reformulation} are also included in Section~\ref{sec:exp}.
 In particular, we highlight the flexibility in the choice of parameters in our aADMM. 

\subsection{Critical Points}

We address the issue of critical points and saddle points of the Lagrangian $L_0$ in \eqref{e:lagrangian}, as well as their relations to the solutions of \eqref{eq:P}, under \cref{as:conv_param}.

\begin{definition}[critical points]
\label{d:crit}
We say that the tuple $(x^\star,z^\star,y^\star)\in\R^n\times\R^m\times\R^m$ is  a \emph{critical point} of the (unaugmented) Lagrangian $L_{0}$ of \eqref{eq:P} if
\begin{equation}\label{eq:critical}
-M^Ty^\star\in\partial f(x^\star),\quad y^\star\in\subd g(z^\star)\quad\text{and}\quad Mx^\star-z^\star=0.
\end{equation}
\end{definition}
We also recall that $(x^\star,z^\star,y^\star)$ is a {\em saddle point} of $L_0$ if
\begin{equation*}
L_0(x^\star,z^\star,y)\leq L_0(x^\star,z^\star,y^\star)\leq L_0(x,z,y^\star),
\quad
\forall (x,z,y)\in\R^n\times \R^m\times \R^m.
\end{equation*}

\begin{lemma}\label{lem:critpoint}
Let \cref{as:conv_param} hold. If $(x^\star,z^\star,y^\star)$ is a critical point of $L_{0}$, then $(x^\star,z^\star)$ is solution of \eqref{eq:P}.
\end{lemma}
\begin{proof}
In view of \eqref{eq:critical}, and by recalling \Cref{f:subdif}\Cref{f:subdif_chain} and \Cref{f:subdif_sum}, we see that
\begin{align*}
0  = -M^Ty^\star+M^Ty^\star 
 & \in  \partial f(x^\star)+M^T\subd g(Mx^\star)\\
 & \subseteq \partial f(x^\star)+\subd(g\circ M)(x^\star)
\subseteq \subd(f+g\circ M)(x^\star).
\end{align*}
Consequently, by combining \Cref{rem:conv_sum} and \Cref{f:subdif}\Cref{f:subdif_convex}, we derive that $(x^\star,z^\star)$ solves \eqref{eq:P}.%
\end{proof}

We see that any critical point produces a solution of \eqref{eq:P}. The converse implication, however, requires a constraint qualification.

\begin{lemma}\label{l:qc-cp}
Suppose that \cref{as:conv_param} and one of the following assertions hold:
\begin{enumerate}[label={\rm(\roman*)},topsep=3pt]
\item\label{as:cp-qc1} $0\in \ri(\dom g - M(\dom f))$;
\item\label{as:cp-qc2} $\ri(\dom g)\cap \ri( M(\dom f))\neq\emptyset$;
\item\label{as:cp-qc3} $\inte(\dom g)\cap M(\dom f)\neq \emptyset\quad$ or $\quad(\dom g)\cap \inte(M(\dom f))\neq \emptyset$.
\end{enumerate}
Then the existence of critical points of $L_0$ is equivalent to the existence of solutions of \eqref{eq:P}.
\end{lemma}
\begin{proof}
$(\Rightarrow)$: Follows from \cref{lem:critpoint}.

$(\Leftarrow)$: We note that either \ref{as:cp-qc2} or \ref{as:cp-qc3} implies \ref{as:cp-qc1} (see, e.g., \cite[Proposition 6.19]{BC17}). Set $\tilde{f}$ and $\tilde{g}$ to be the convex functions in \cref{rem:reformulation}. We observe that $\dom f=\dom \tilde{f}$, $\dom g=\dom\tilde{g}$, and that
\begin{equation*}
\partial \tilde{f} = \partial f +\beta M^TM
\quad,\quad
\partial \tilde{g} = \widehat{\partial} g -\beta\Id.
\end{equation*}
Since the constraint qualification $0\in\ri(\dom \tilde g - M(\dom \tilde f))$ is satisfied, by employing subdifferential calculus (see \cite[Theorem~16.47]{BC17}) we arrive at
\begin{align*}
\partial(f+g\circ M)&=\partial(\tilde{f}+\tilde{g}\circ M)\\
&=\partial \tilde{f} + M^T\circ\partial\tilde{g}\circ M\\
&=\partial f +\beta M^TM + M^T\circ\widehat{\partial} g\circ M
-\beta M^TM\\
&=\partial f + M^T\circ\widehat{\partial} g\circ M.
\end{align*}
Finally, by invoking \Cref{rem:conv_sum}, if $(x^\star,z^\star)$ is a solution of \eqref{eq:P}, then 
$$z^\star=Mx^\star \quad \text{and}\quad 0\in\partial(f+g\circ M)(x^\star)=\partial f(x^\star)+M^T\circ\widehat{\partial}g\circ M(x^\star).$$ 
Consequently, there exists $y^\star\in\R^m$ such that $(x^\star,z^\star,y^\star)$ is a critical point of $L_0$.%
\end{proof}

The notion of a critical point and the one of a saddle point coincide in the case where both functions $f$ and $g$ are convex. We now show that in the case where convexity is absent, saddle points are still critical points.

\begin{lemma}[critical points vs saddle points]
\label{l:cp-sp}
Let {$f:\R^n\to\left]-\infty,\infty\right]$} and {$g:\R^m\to\left]-\infty,\infty\right]$} be proper. Then, the saddle points of $L_0$ are also critical points. If, in addition, $f$ and $g$ are convex, then any critical point of $L_0$ is a saddle point.
\end{lemma}
\begin{proof}
Let $x^\star\in\dom(f),\ z^\star\in\dom(g)$. Then
\begin{subequations}
\begin{alignat}{4}
f(x^\star)+&&g(z^\star)+\langle y, Mx^\star-z^\star\rangle &\leq f(x^\star)+g(z^\star)+\langle y^\star,Mx^\star-z^\star\rangle,  \quad\forall y\in\R^m, \label{eq:sadp1}\\
\iff\quad&&
\langle y-y^\star, Mx^\star-z^\star\rangle &\leq 0, \quad \forall y\in\R^m,\\
\iff\quad && Mx^\star&=z^\star.
\end{alignat}
\end{subequations}
Moreover,
\begin{align}
f(x^\star)+g(z^\star)+\langle y^\star,Mx^\star-z^\star\rangle & \leq f(x)+g(z)+\langle y^\star, Mx-z\rangle,\label{eq:sadp2}
\end{align}
for all $(x,z)\in\R^n\times\R^m$, is equivalent to
\begin{align*}
f(x^\star)+\langle y^\star,Mx^\star\rangle & \leq f(x)+\langle y^\star,Mx\rangle, \quad \forall x\in\R^n,\\
\text{and} \qquad\quad g(z^\star)-\langle y^\star,z^\star\rangle & \leq g(z)-\langle y^\star,z\rangle, \quad \forall z\in\R^m,
\end{align*}
i.e.,
\begin{equation*}
-M^T y^\star\in\partial f(x^\star)
\quad\text{and}\quad y^\star\in\partial g(z^\star)\subseteq \widehat{\partial} g(z^\star).
\end{equation*}
Hence, if $(x^\star,z^\star,y^\star)$ is a saddle point, then $(x^\star,z^\star,y^\star)$ is a critical point of $L_0$.

Conversely, if, in addition, $f$ and $g$ are convex and if $(x^\star,z^\star,y^\star)$ is a critical point of $L_0$, then
\begin{equation*}
Mx^\star=z^\star\ ,\ -M^T y^\star\in\partial f(x^\star)
\quad\text{and}\quad y^\star\in\widehat{\partial} g(z^\star)=\partial g(z^\star),
\end{equation*}
which implies \eqref{eq:sadp1} and \eqref{eq:sadp2}, i.e., $(x^\star,z^\star,y^\star)$ is a saddle point of $L_0$.%
\end{proof}

In view of the relations between the critical points of $L_0$ and the solutions of \eqref{eq:P}, we impose the existence of a critical point in our convergence analysis.

\begin{assumption}\label{as:critpoint}
The Lagrangian $L_0$ has a critical point.
\end{assumption}

\begin{remark}\label{r:assumptions_cp} \Cref{as:critpoint} is standard in the analysis of the ADMM and its variants in the convex framework. In this case \Cref{l:cp-sp} implies that it is equivalent to the existence of saddle points, which is assumed in several classical studies such as~\cite{BPCPE11,EY15,FG83}.

Other authors obtain the existence of critical/saddle points from the nonemptyness of the solution set of \eqref{eq:P} when combined with one of the constraint qualifications in \Cref{l:qc-cp}. For instance, the \emph{Slater constraint qualification} in \Cref{l:qc-cp}\ref{as:cp-qc2} is used in \cite{CST17,CP11} while \cite{BotCset19,MZ19} incorporate the assumption in \Cref{l:qc-cp}\ref{as:cp-qc1}.
\end{remark}

We now relate the critical points of the Lagrangian to the zeros of $Q+S$, where $Q$ and $S$ are the operators defined by
\begin{subequations}\label{e:QS}
\begin{align}
Q:\R^m\tto\R^m: y\mapsto &\left\{ -Mx : -M^Ty \in \partial f(x)\right\}
=(-M)\circ(\partial f)^{-1}\circ (-M^T)(y),\label{e:Q}\\
S:\R^m\tto\R^m: y\mapsto &\left\{ z : y \in \subd g(z)\right\}=(\subd g)^{-1}(y).\label{e:S}
\end{align}
\end{subequations}
We will address the convergence of our aADMM by applying the adaptive DR algorithm~\cite{DPadapt} to $Q$ and $S$. This is a natural extension of the classical relation between the ADMM and the DR algorithm in the convex case (see, e.g., \cite{EY15,MZ19}) to our generalized settings.

\begin{proposition}\label{prop:zerQS}
The Lagrangian $L_0$ has a critical point if and only if $\zer(Q+S)\neq\emptyset$, where $Q$ and $S$ are the operators defined in \eqref{e:QS}. More precisely, $y^\star\in\zer(Q+S)$ if and only if there exist $x^\star\in\R^n$ and $z^\star\in\R^m$ such that $(x^\star,z^\star,y^\star)$ is a critical point of $L_0$.\label{prop:zerQS_II}
\end{proposition}
\begin{proof}
We observe that $y^\star\in\zer(Q+S)$ if and only if there exists $z^\star\in\R^m$ such that
\begin{equation}\label{e:zer Q and S}
z^\star\in S(y^\star) \quad\text{and}\quad -z^\star\in Q(y^\star).
\end{equation}
The definition of $Q$ and $S$ implies that \eqref{e:zer Q and S} is equivalent to $y^\star\in\subd g(z^\star)$ and the existence of $x^\star\in\R^n$ such that $z^\star=Mx^\star$ and $-M^Ty^\star\in\partial f(x^\star)$, that is, $(x^\star,z^\star,y^\star)$ is a critical point of~$L_0$.%
\end{proof}

\subsection{The Algorithm}

We now formulate our adaptive version of the Alternating Direction Method of Multipliers, \emph{aADMM} for short. The steps of the aADMM are analogous to the steps of the classical ADMM in \eqref{alg0}, however, the aADMM is admissible in the strongly-weakly convex setting and it accommodates different penalty parameters in the two minimization steps. Specifically, we set an initial point $(x^0,z^0,y^0)\in\R^n\times\R^m\times\R^m$ and two parameters $\gamma,\delta>0$. Then the aADMM iterates according to the recurrences
\begin{subequations}\label{alg}
\begin{align}
x^{k+1}&=\argmin_{x\in \R^n}L_\gamma(x,z^k,y^k),
\label{alg:x}\\
z^{k+1}&=\argmin_{z\in \R^m}L_\delta(x^{k+1},z,y^k),
\label{alg:z}\\
y^{k+1}&=y^k+\delta (Mx^{k+1}-z^{k+1}),\label{alg:lm}
\end{align}
\end{subequations}
equivalently,
\begin{subequations}\label{alg2}
\begin{align}
x^{k+1}&=\argmin_{x\in \R^n}\left\{f(x) + \frac{\gamma}{2}\Big\|Mx-z^k+\frac{y^k}{\gamma}\Big\|^2\right\},\label{alg2:x}\\
z^{k+1}&=\argmin_{z\in \R^m}\left\{g(z)+ \frac{\delta}{2}\Big\|Mx^{k+1}-z+\frac{y^k}{\delta}\Big\|^2\right\},\label{alg2:z}\\
y^{k+1}&=y^k+\delta (Mx^{k+1}-z^{k+1}).\label{alg2:lm}
\end{align}
\end{subequations}

Clearly, by letting $\gamma=\delta$, we obtain the steps of the original ADMM. Similar to the ADMM, the aADMM is only valid if the $x^k$-step in \eqref{alg:x}  and the $z^k$-step in \eqref{alg:z} (equivalently, \eqref{alg2:x} and \eqref{alg2:z}, respectively) are well defined. We will examine this issue in relation to the operators $Q$ and $S$ in \eqref{e:QS}. Our next argument follows the footsteps of \cite{Gab83} (see also \cite{EY15}). It forms a foundation for convergence analysis of the aADMM by providing a sufficient condition for the existence of the $z$-update via the maximal comonotonicity of $S$.

\begin{lemma}[existence of the $z$-update]
\label{l:zk_exists}
Let {$g:\R^m\to\left]-\infty,\infty\right]$} be proper, $\beta$-convex and lower semicontinuous. Let $x^{k+1}\in\R^n$, $y^k\in\R^m$ and $\delta>\max\{0,-\beta\}$. Then the operator $S$ defined by \eqref{e:S} is maximally $\beta$-comonotone. Consequently, $J_{\delta S}$ is single-valued with full domain and $z^{k+1}$ defined in \eqref{alg:z} is uniquely determined by
\begin{equation*}
z^{k+1}=\frac{1}{\delta}\big(\Id-J_{\delta S}\big)(y^k+\delta Mx^{k+1}),
\end{equation*}
and
\begin{equation*}
y^{k+1}=y^k+\delta(Mx^{k+1}-z^{k+1})= J_{\delta S}(y^k+\delta Mx^{k+1}).
\end{equation*}
\end{lemma}
\begin{proof}
Since $g$ is $\beta$-convex, \Cref{f:subd_maxmon} implies that $\subd g$ is maximally $\beta$-monotone. Consequently, it follows from \Cref{f:maxcom}\Cref{f:maxcom_I} that $S=(\subd g)^{-1}$ is maximally $\beta$-comonotone.

Now, since $\delta>-\beta$, \cref{f:resolvent_comono} implies that $J_{\delta S}$ is single-valued and has full domain. Furthermore, since $g$ is $\beta$-convex and $\delta>-\beta$, the function inside the argmin in \eqref{alg:z} is convex. By employing \cref{f:subdif}~\ref{f:subdif_minc} and then \ref{f:subdif_sumdif}, we see that
\begingroup
\allowdisplaybreaks
\begin{align*}
z^{k+1}\text{ satisfies }\eqref{alg:z}
\quad\iff\quad & 0\in\subd g(z^{k+1})-y^k-\delta(Mx^{k+1}-z^{k+1})\\
\iff\quad & y^{k+1}\in \subd g(z^{k+1})\\
\iff\quad & z^{k+1}\in S(y^{k+1})\\
\iff\quad & y^{k+1}+\delta z^{k+1}\in (\Id+\delta S)(y^{k+1})\\
\iff\quad & y^{k+1}= J_{\delta S}(y^{k+1}+\delta z^{k+1}) =J_{\delta S}(y^k+\delta Mx^{k+1})
\end{align*}
\endgroup
which completes the proof.%
\end{proof}

We now provide a general condition for the existence of the $x$-update. We dedicate \Cref{sec:qcxup} to a detailed  discussion of cases where this condition is satisfied.

\begin{lemma}\label{l:Qcomono}
Let $M\in\R^{m\times n}$ be nonzero and let {$f:\R^n\to\left]-\infty,\infty\right]$} be proper, lower semicontinuous and $\alpha$-convex, where $\alpha\in\R_+$. Then the operator $Q$ defined in \eqref{e:Q} is $\frac{\alpha}{\|M\|^2}$-comonotone.
\end{lemma}
\begin{proof}
Let $(y_1,-Mx_1),(y_2,-Mx_2)\in\gra Q$. Then $(x_1,-M^Ty_1),(x_2,-M^Ty_2)\in\gra \partial f$. Since $f$ is $\alpha$-convex, \Cref{f:subd_maxmon} implies that $\partial f$ is $\alpha$-monotone. Consequently,
\begin{equation*}
\langle x_1-x_2, -M^T(y_1-y_2)\rangle \geq \alpha\|x_1-x_2\|^2
\end{equation*}
which implies that
\begin{equation*}
\langle -M(x_1-x_2), y_1-y_2\rangle \geq \frac{\alpha}{\|M\|^2}\|M\|^2\|x_1-x_2\|^2\geq\frac{\alpha}{\|M\|^2}\|M(x_1-x_2)\|^2.
\end{equation*}
Thus, $Q$ is $\frac{\alpha}{\|M\|^2}$-comonotone.%
\end{proof}

\begin{lemma}[conditions for the existence of the $x$-update]
\label{l:xk_exists}
Let {$f:\R^n\to\left]-\infty,\infty\right]$} be proper, convex and lower semicontinuous. Let $M\in\R^{m\times n}$ be nonzero. Let $y^k,z^k\in\R^m$ and $\gamma>0$. Then the following assertions are equivalent.
\begin{enumerate}
\item\label{l:xk_exists-i} $x^{k+1}$ satisfies \eqref{alg:x};
\item\label{l:xk_exists-ii} $x^{k+1}\in(\partial f)^{-1}(-M^Tv^k)$, where $v^k=y^k+\gamma(Mx^{k+1}-z^k)$;
\item\label{l:xk_exists-iii} $v^k= J_{\gamma Q}(y^k-\gamma z^k)$, where $Q$ is defined by \eqref{e:Q}.
\end{enumerate}
Consequently, $x^{k+1}$ in \eqref{alg:x} exists for all $(y^k,z^k)$ if and only if $J_{\gamma Q}$ has full domain. 
\end{lemma}
\begin{proof}
By invoking \Cref{f:subdif}~\ref{f:subdif_minc} and~\ref{f:subdif_sumdif}, we see that
\begin{align}
&\text{$x^{k+1}$ satisfies \eqref{alg:x}}~~\nonumber\\
\iff~~& 0\in\partial f(x^{k+1})+M^Ty^k+\gamma M^T(Mx^{k+1}-z^k)\nonumber\\
\iff~~& -M^Tv^k\in \partial f(x^{k+1})
    \text{~~where~~}v^k=y^k+\gamma(Mx^{k+1}-z^k)\nonumber\\
\iff~~& x^{k+1}\in(\partial f)^{-1}(-M^Tv^k)\text{~~where~~}
    y^k-\gamma z^k-v^k=-\gamma Mx^{k+1}\nonumber\\
\iff~~& y^k-\gamma z^k -v^k\in -\gamma M\circ (\partial f)^{-1}\circ (-M^T)(v^k)\nonumber\\
\iff~~& y^k-\gamma z^k \in (\Id +\gamma Q)(v^k)\nonumber\\
\iff~~&v^k = J_{\gamma Q}(y^k-\gamma z^k).\label{eq:JgammaQ single valued}
\end{align}
In~\eqref{eq:JgammaQ single valued} we employed the single-valuedness of $J_{\gamma Q}$ which follows from \Cref{l:Qcomono} when combined with \Cref{f:resolvent_comono}\ref{f:resolvent_comono_I}.%
\end{proof}

\subsection{Constraint Qualifications for the $x$-update}\label{sec:qcxup}

While the $z$-update is already {well defined} under a generalized convexity assumption (by \cref{l:zk_exists}), the $x$-update depends on the resolvent of $Q$ having full domain (by \cref{l:xk_exists}). We now discuss constraint qualifications for the maximal comonotonicity of $Q$, which, in turn, guarantees full domain of its resolvent. To this end we assume that

\begin{assumption}\label{as:QC}
The following constraint qualification holds:
\begin{equation}\label{eq:QC}
0\in\ri(\dom f^* - \ran M^T).
\end{equation}
\end{assumption}

\cref{as:QC} is satisfied in several cases which we detail in the following lemma. Some of the historical context and references to  such cases in the classical ADMM literature are provided in \cref{rem:QC history}. 

\begin{lemma}[sufficient conditions for \cref{as:QC}]
\label{l:QCequiv}
Let $f:\R^n\to\left]-\infty,+\infty\right]$ be proper, convex and lower semicontinuous. Let $M\in\R^{m\times n}$. Then each of the following conditions implies~\eqref{eq:QC}:
\begin{enumerate}
\item\label{l:QCequiv-i} $\ri(\dom f^*)\cap \ran M^T \neq\varnothing$.
\item\label{l:QCequiv-ii} $\ri(\ran \partial f)\cap \ran M^T \neq\varnothing$.
\item\label{l:QCequiv-iii} $(\rec f)(x)>0$ for all $x\in \ker M\setminus\{x\in\R^n : -(\rec f)(-x)=(\rec f)(x)=0\}.$
\item\label{l:QCequiv-iv} $f$ is coercive (in particular, supercoercive).
\item\label{l:QCequiv-ivb} $f$ is strongly convex.
\item\label{l:QCequiv-v} $M^TM$ is invertible.
\end{enumerate}
\end{lemma}
\begin{proof}
\ref{l:QCequiv-i}:
By \cite[Corollary~6.6.2]{Rock72}, 
 $$ 
 \ri(\dom f^* - \ran M^T)=\ri(\dom f^*) - \ran M^T.
 $$ 
Consequently, \ref{l:QCequiv-i} follows from \cite[Corollary~6.15]{BC17}.

\ref{l:QCequiv-ii}: By invoking \ref{l:QCequiv-i} , it suffices to show that
\begin{equation}\label{l:QCequiv-e1}
\ri(\dom f^*)=\ri(\ran\partial f).
\end{equation}
Indeed, by \cite[Proposition~16.4(i) and Corollary 16.18(i)]{BC17}, 
\begin{equation*}
\ri (\dom f^*)\subseteq \dom \partial f^*\subseteq \dom f^*.
\end{equation*}
Now, by employing \cite[Theorem~6.3]{Rock72}, 
\begin{equation*}
\cl(\dom f^*)=\cl(\dom \partial f^*)
\quad\text{and}\quad
\ri(\dom f^*)=\ri(\dom\partial f^*).
\end{equation*}
Finally, since $\dom \partial f^*=\ran\partial f$, we arrive at \eqref{l:QCequiv-e1}.

\ref{l:QCequiv-iii}: We consider the function $F=f+h$, where $h=\eta\|M(\cdot)-a\|^2$ for some $a\in\R^m$ and $\eta>0$. By \cite[Proposition~9.30(vi)]{BC17},
$$\rec F = \rec f +\rec h.$$
Moreover, since $\|\cdot\|^2$ is supercoercive, \cite[Proposition~9.30(vii) and Example~9.32]{BC17} imply that
$$(\rec h)(x):=\left\{\begin{array}{ll}
0, & \text{if } Mx=0,\\
+\infty, & \text{otherwise}.\end{array}\right.$$
Hence, \ref{l:QCequiv-iii} implies that $(\rec F)(x)>0$ for all vectors $x\in\R^n$ except those satisfying $-\rec F(-x)=\rec F(x)=0$. By \cite[Corollary~13.3.4(b)]{Rock72}, this is equivalent to
\begin{equation}\label{ax08311809a}
0\in\ri(\dom F^*).
\end{equation}
On the other hand, since $\dom h^* = \ran M^T$, by employing \cite[Proposition~12.6(ii)]{BC17} we see that 
\begin{equation*}
\dom f^* - \ran M^T = \dom f^* + \ran M^T
= \dom f^* +\dom h^* = \dom(f^*\square h^*).
\end{equation*}
Now, since $h$ has full domain, $0\in\inte(\dom f-\dom h)$. Consequently, by \cite[Theorem~11.23(a)]{RW} we arrive at $f^*\square h^*=(f+h)^*$ which, in turn, implies that
\begin{equation}\label{ax08311809b}
\dom f^* - \ran M^T =  \dom(f^*\square h^*)=\dom (f+h)^*=\dom F^*.
\end{equation}
By combining \eqref{ax08311809a} and \eqref{ax08311809b} we obtain \eqref{eq:QC}.

\ref{l:QCequiv-iv}: If $f$ is coercive, then \cite[Proposition~14.16]{BC17} implies that $0\in\inte(\dom f^*)\subseteq\ri(\dom f^*)$. Since, clearly, $0\in\ran M^T$, we employ \ref{l:QCequiv-i} and conclude \eqref{eq:QC}.

\ref{l:QCequiv-ivb}: Follows from \ref{l:QCequiv-iv} since every strongly convex function is coercive.

\ref{l:QCequiv-v}: If $M^TM$ is invertible, then $\ran M^T=\R^n$. Since $\dom f^*\neq\varnothing$, \eqref{eq:QC} follows trivially.%
\end{proof}

We now prove the existence of the $x$-update via maximal comonotonicity of $Q$ which is guaranteed by \cref{as:conv_param,as:QC}.

\begin{lemma}[existence of the $x$-update under constraint qualifications]
\label{p:Qmax}
Let $M\in\R^{m\times n}$ be nonzero and let {$f:\R^n\to\left]-\infty,\infty\right]$} be proper, lower semicontinuous, and $\alpha$-convex for some $\alpha\in\R_+$. 
Let \cref{as:QC} hold. Then the operator $Q$ defined in \eqref{e:Q} is maximally $\frac{\alpha}{\|M\|^2}$-comonotone. Consequently, the $x$-update in \eqref{alg:x} is {well defined} and
\begin{equation*}
y^k+\gamma(Mx^{k+1}-z^k) = J_{\gamma Q}(y^k-\gamma z^k).
\end{equation*}
\end{lemma}
\begin{proof}

Since $f$ is convex, proper and lsc, so is $f^*$ (see, e.g., \cite[Proposition~13.13]{BC17}). Consequently, since \Cref{as:QC} holds, the constraint qualifications in the chain rule \cite[Theorem~23.9]{Rock72} (see also \cite[Corollary~16.53]{BC17}) are met for the convex function $f^*$ and the linear operator $-M^T$ and we obtain
\begin{equation}\label{eq:chainruleQ}
\partial (f^*\circ (-M^T))= (-M)\circ (\partial f^*) \circ (-M^T)
= (-M)\circ (\partial f)^{-1} \circ (-M^T) = Q.
\end{equation}
We see that $Q$ is the subdifferential of the proper, convex and lower semicontinuous function $f^*\circ (-M^T)$ and, as such, $Q$ is maximally monotone by \Cref{f:subd_maxmon}. It now follows that $Q$ is maximally $\frac{\alpha}{\|M\|^2}$-comonotone by \cref{l:Qcomono} and \cref{f:maxcom}\ref{f:maxcom_II}.

Finally, \cref{f:resolvent_comono} implies that $J_{\gamma Q}$ is single-valued with full domain. The existence of the $x$-update now follows from \cref{l:xk_exists}.%
\end{proof}

\begin{remark}\label{rem:QC history} \cref{as:QC} and several of the constraint qualification in \Cref{l:QCequiv}  have been widely used for the analysis of the classical ADMM in the literature. In this relation, we list some classical and recent references:
\begin{enumerate}
\item In \cite{Gab83} the convergence of the ADMM for the
variational inequality problem
\begin{align*}
&\text{find $x,w\in\R^n$ such that}\\
&w\in A(x)\quad\text{and}\quad \langle w,y-x\rangle +g(My)-g(Mx)\geq 0,\quad\forall y\in\R^n;
\end{align*}
was established. One of the assumptions is that either $A$ is strongly monotone or $M^TM$ is an isomorphism (see \cite[Theorem~5.1]{Gab83}).
When applied to problem \eqref{eq:P}, these assumptions become
\begin{equation*}
\text{$f$ is strongly convex\quad or\quad $M^TM$ is invertible},
\end{equation*}
which implies \cref{as:QC} by \Cref{l:QCequiv}~\ref{l:QCequiv-ivb} and \ref{l:QCequiv-v}.

\item
\cite{BPCPE11} is one of the most widely cited studies of the ADMM. Therein, no conditions are imposed (apart from the existence of saddle points; see~\Cref{r:assumptions_cp}) in order to derive convergence of the ADMM. However, it was pointed out in \cite{CST17} that this may fail since in this case the steps of the ADMM may not be {well defined}, i.e., the argmins may not exist. In particular, for the existence of $x^{k+1}$, the authors of \cite{CST17} require the condition in \cref{l:QCequiv}\ref{l:QCequiv-iii} (see \cite[Assumption~1]{CST17}). 

\item In \cite{RLY19}, the condition in \cref{l:QCequiv}\ref{l:QCequiv-i} was imposed to guarantee the existence of $x^{k+1}$.

\item In \cite{EckThesis}, the condition in \Cref{l:QCequiv}\ref{l:QCequiv-ii} was utilized in order to obtain the chain rule in \eqref{eq:chainruleQ}.
It was referred to as {\em dual normality} (see \cite[Definition~3.22 and Proposition~3.30]{EckThesis}). The argument therein also follows directly from \cite[Theorem~23.9]{Rock72}. The chain rule leads to the maximal monotonicity of $Q$, which, in turn, guarantees the existence of $x^{k+1}$.

\item The invertibility of $M^TM$  in \Cref{l:QCequiv}\ref{l:QCequiv-v} is, arguably, the most common assumption for the convergence of ADMM in the literature. Indeed, it is assumed in \cite{ACL16,BotCset19,CP11,FG83,MZ19}, to name a few.
\end{enumerate}
\end{remark}

Although $M^TM$ being invertible or $f$ being strongly convex are two of the more restrictive conditions in \cref{l:QCequiv}, as we show next, they do provide an additional strength: uniqueness of the $x$-update which, in turn, implies convergence of the sequence $\seq{x^k}{k}$. \cref{as:QC} alone only guarantees the convergence of the sequence $\seq{Mx^k}{k}$ (see \cref{t:cvg_aadmm}).

\begin{lemma}[criteria for  uniqueness of the $x$-update]\label{l:xk_unique}
Let $M\in\R^{m\times n}$ be nonzero and let {$f:\R^n\to\left]-\infty,\infty\right]$} be proper, lower semicontinuous and $\alpha$-convex, where $\alpha\in\R_+$.  Set
\begin{equation*}
v^k:= J_{\gamma Q}(y^k-\gamma z^k)\ \text{where}\ Q\ \text{is the operator defined~in~\eqref{e:Q}}.
\end{equation*}
The following hold. 
\begin{enumerate}
\item \label{l:xk_unique_f} If $\alpha>0$ (i.e. $f$ is strongly convex), then $(\partial f)^{-1}$ is Lipschitz continuous and the $x$-update in \eqref{alg:x} is uniquely defined by
\begin{equation*}
x^{k+1}=(\partial f)^{-1}(-M^Tv^k).
\end{equation*}
\item \label{l:xk_unique_M} If $M^TM$ is invertible, then the $x$-update in \eqref{alg:x} is uniquely defined by
\begin{equation*}
x^{k+1}=(M^TM)^{-1}M^T\left(\tfrac{1}{\gamma}(v^k-y^k)+z^k\right).
\end{equation*} 
\end{enumerate}
\end{lemma}
\begin{proof}
By employing \Cref{l:QCequiv}\ref{l:QCequiv-ivb} and \Cref{l:QCequiv}\ref{l:QCequiv-v}, respectively, we see that the conditions in \ref{l:xk_unique_f} and \ref{l:xk_unique_M} imply \Cref{as:QC}. Consequently, \Cref{p:Qmax} implies that the $x$-update in \eqref{alg:x} is {well defined} and
\begin{equation*}
v^k:=y^k+\gamma(Mx^{k+1}-z^k) = J_{\gamma Q}(y^k-\gamma z^k).
\end{equation*}
Now, if $f$ is $\alpha$-strongly convex ($\alpha>0$), then by combining \Cref{f:subd_maxmon} with \Cref{f:maxcom}\ref{f:maxcom_I} we see that $(\partial f)^{-1}$ is maximally $\alpha$-comonotone. Consequently, $(\partial f)^{-1}$ is Lipschitz continuous, in particular, single-valued. It follows that the inclusion in \Cref{l:xk_exists}\ref{l:xk_exists-ii} is an equality, i.e.,
\begin{equation*}
x^{k+1}=(\partial f)^{-1}(-M^Tv^k).
\end{equation*}
If $M^TM$ is invertible, since $v^k=y^k+\gamma(Mx^{k+1}-z^k)$, we conclude that
\begin{equation*}
x^{k+1}=(M^TM)^{-1}(M^TM)x^{k+1}=(M^TM)^{-1}M^T\left(\tfrac{1}{\gamma}(v^k-y^k)+z^k\right),
\end{equation*}
which completes the proof.%
\end{proof}

\begin{remark}
We note that \Cref{l:xk_unique} is directly implied by~\eqref{alg2:x}. Indeed, if $f$ is strongly convex or $M$ has full column rank (i.e. $M^TM$ is invertible), then the function inside the argmin in~\eqref{alg2:x} is strongly convex, which, in turn, implies existence and uniqueness of minimizers.
\end{remark}

\section{Convergence of the aADMM}\label{sec:conv}

In order to obtain convergence of the aADMM, we adapt $\gamma$ and $\delta$ to the convexity parameters of the two functions $f$ and $g$. We will employ the following lemma in order to guarantee existence of such parameters.

\begin{lemma}[existence of parameters]\label{l:param}
Let $\alpha,\beta\in\R$ such that $\alpha+\beta>0$. Then for every $\gamma,\delta\in\R_{++}$ the following assertions are equivalent.
\begin{enumerate}
\item $2\delta(\alpha+\beta)+(\gamma+\delta)^2<4(\gamma+\alpha)(\delta+\beta)$.
\item $\delta+2\beta>0$ and $$\gamma\in\left]\delta+2\beta-\sqrt{2(\alpha+\beta)(\delta+2\beta)},\delta+2\beta+\sqrt{2(\alpha+\beta)(\delta+2\beta)}\right[.$$
\end{enumerate}
\end{lemma}
\begin{proof}
 A straightforward computation implies that
\begin{align*}
& 2\delta(\alpha+\beta)+(\gamma+\delta)^2<4(\gamma+\alpha)(\delta+\beta)\\
\iff\quad & \gamma^2-2(\delta+2\beta)\gamma+
(\delta+2\beta)(\delta-2\alpha)<0\\
\iff\quad & \left\{\begin{array}{l}
\delta+2\beta>0,\\
\delta+2\beta-\sqrt{2(\alpha+\beta)(\delta+2\beta)}<\gamma<\delta+2\beta+\sqrt{2(\alpha+\beta)(\delta+2\beta)};\end{array}\right.
\end{align*}
which completes the proof.%
\end{proof}

\begin{theorem}[convergence of the aADMM]
\label{t:cvg_aadmm}
Suppose that Assumptions \ref{as:conv_param}, \ref{as:critpoint} and \ref{as:QC} hold. Let $\delta>\max\{0,-2\beta\}$ and set
\begin{subequations}\label{eq:ass_param}
\begin{align}
&\gamma=\delta+2\beta, & \text{if } \alpha+\beta\|M\|^2=0,\label{eq:ass_paramI}\\
&\gamma\in\left]\max\{0,\delta+2\beta - \Delta_{\delta}\}, \delta+2\beta + \Delta_{\delta} \right[, & \text{if } \alpha+\beta\|M\|^2>0;\label{eq:ass_paramII}
\end{align}
\end{subequations}
where 
$$\Delta_{\delta}:=\frac{1}{\|M\|}\sqrt{2\left({\alpha}+\beta{\|M\|^2}\right)(\delta+2\beta)}.$$ 
Set $(x^0,z^0,y^0)\in\R^n\times\R^m\times\R^m$ and let $\seq{x^k,z^k,y^k}{k}$ be generated by the aADMM \eqref{alg}. Then
\begin{equation*}
Mx^k\to Mx^\star,\quad z^k\to z^\star \quad\text{and}\quad y^k\to y^\star,
\end{equation*}
where $(x^\star,z^\star,y^\star)$ is a critical point of $L_0(x,z,y)$. Consequently, $(x^\star,z^\star)$ is a solution of \eqref{eq:P}.
\end{theorem}
\begin{proof}
Let $Q$ and $S$ be defined by \eqref{e:QS}. Our aim is to establish that the sequence generated by the aADMM \eqref{alg} corresponds to the sequence generated by the aDR~\eqref{eq:aDR_iter} when applied to $Q$ and $S$. Note that $\zer(Q+S)\neq\emptyset$ due to \Cref{as:critpoint,prop:zerQS}.

On the one hand, for any $\beta$, it holds that $\max\{0,-2\beta\}\geq \max\{0,-\beta\}$. So $\delta>\max\{0,-\beta\}$. Consequently, \cref{l:zk_exists} implies that $S$ is maximally $\beta$-comonotone and that the $z$-update in \eqref{alg:z} is uniquely defined by
\begin{equation*}
y^{k+1}= J_{\delta S}(y^{k+1}+\delta z^{k+1}).
\end{equation*}
On the other hand, \cref{p:Qmax} implies that $Q$ is maximally $\frac{\alpha}{\|M\|^2}$-comonotone, that the $x$-update in \eqref{alg:x} is well defined and
\begin{equation*}
v^k:=y^k+\gamma(Mx^{k+1}-z^k)= J_{\gamma Q}(y^k-\gamma z^k).
\end{equation*}
Set $w^k:=y^k+\delta z^k$ for each $k=0,1,2,\ldots$. Set further
\begin{equation*}
\lambda:=1+\frac{\delta}{\gamma}
\quad\text{and}\quad
\mu:=1+\frac{\gamma}{\delta}.
\end{equation*}
Then, clearly, $\lambda,\mu$ satisfy \eqref{e:lm_mu}. We employ $J_{\delta S}$ and $J_{\gamma Q}$ in order to compute $J^{\lambda}_{\gamma Q}\big(J^{\mu}_{\delta S}(w^k)\big)$ via the following steps:
\begin{align*}
J_{\delta S}(w^k)&=J_{\delta S}(y^k+\delta z^k)=y^k,\\
J^{\mu}_{\delta S}(w^k)&=(1-\mu)w^k+\mu J_{\delta S}(w^k)\\
&=(1-\mu)(y^k+\delta z^k)+\mu y^k\\
&=y^k+(1-\mu)\delta z^k \\
&=y^k-\gamma z^k,\\
J_{\gamma Q}\big(J^{\mu}_{\delta S}(w^k)\big)&=J_{\gamma Q}(y^k-\gamma z^k)=v^k,\\
\text{and so,}\quad\ J^{\lambda}_{\gamma Q}\big(J^{\mu}_{\delta S}(w^k)\big)&=(1-\lambda)J^{\mu}_{\delta S}(w^k)+\lambda J_{\gamma Q}\big(J^{\mu}_{\delta S}(w^k)\big)\\
&=(1-\lambda)(y^k-\gamma z^k)+\lambda v^k\\
&=y^k+\delta z^k+\lambda(v^k-y^k)\\
&=w^k+\lambda(v^k-y^k).
\end{align*}
We also observe that
\begin{align*}
w^{k+1}&=y^{k+1}+\delta z^{k+1}=y^k+\delta Mx^{k+1}=y^k+(\lambda-1)\gamma Mx^{k+1}\\
&=y^k+(\lambda-1)(v^k-y^k+\gamma z^k)=y^k+(\lambda-1)(v^k-y^k)+\delta z^k\\
&=w^k+(\lambda-1)(v^k-y^k).
\end{align*}
Consequently, by setting $\kappa:=\frac{\lambda-1}{\lambda}=\frac{\delta}{\gamma+\delta}\in{\left]0,1\right[}$ and by recalling~\eqref{eq:aDR_iter} we arrive at
\begin{equation}\label{eq:dual_aDR}
\Tdr(w^k):=(1-\kappa)w^k+\kappa J^{\lambda}_{\gamma Q}\big(J^{\mu}_{\delta S}(w^k)\big)=w^k+(\lambda-1)(v^k-y^k)=w^{k+1}.
\end{equation}
Summing up, the sequence $\seq{w^k}{k}$ is generated by the aDR algorithm applied to $S$~and~$Q$. 

We now apply \cref{f:aDR_comono} to the two operators $Q$ and $S$. To this end, it suffices to check that the conditions in \cref{f:aDR_comono} for the parameters, i.e., conditions \eqref{e:aDR_comono} and \eqref{e:aDR_kappa} for $\frac{\alpha}{\|M\|^2},\beta,\gamma,\delta$, and $\kappa$, are met. By \Cref{as:conv_param},
\begin{equation*}
\frac{\alpha}{\|M\|^2}+\beta\geq 0.
\end{equation*}

If $\frac{\alpha}{\|M\|^2}+\beta=0$, then \eqref{eq:ass_paramI} implies
$\delta=\gamma-2\beta=\gamma+2\frac{\alpha}{\|M\|^2}$,
that is, \eqref{eq:assumtions_comonoI} is satisfied.

If $\frac{\alpha}{\|M\|^2}+\beta>0$, then, by \Cref{l:param}, we see that \eqref{eq:ass_paramII} implies that
\begin{subequations}
\begin{align}
&&2\delta\left(\frac{\alpha}{\|M\|^2}+\beta\right)+(\gamma+\delta)^2&<4\left(\gamma+\frac{\alpha}{\|M\|^2}\right)(\delta+\beta)
\label{e:200705a}\\
\iff&& 
\frac{\delta}{\gamma+\delta} &< \frac{4\left(\gamma+\frac{\alpha}{\|M\|^2}\right)(\delta+\beta)-(\gamma+\delta)^2}{2(\gamma+\delta)\left(\frac{\alpha}{\|M\|^2}+\beta\right)}=:\overline{\kappa}.
\label{e:200705b}
\end{align}
\end{subequations}
On the one hand, \eqref{e:200705a} implies 
\begin{equation*}
(\gamma+\delta)^2<4\left(\gamma+\frac{\alpha}{\|M\|^2}\right)(\delta+\beta),
\end{equation*}
which leads to \eqref{eq:assumtions_comonoII}. On the other hand, \eqref{e:200705b} implies that
$0<\kappa<\overline{\kappa}$, i.e., \eqref{e:aDR_kappa} is satisfied.

Consequently, we meet all of the conditions in order to apply \Cref{f:aDR_comono}: \Cref{f:aDR_comono}\Cref{f:aDR_comonoI} implies that $w^k\to w^\star\in \Fix\Tdr$ while \Cref{f:aDR_comono}\Cref{f:aDR_comonoII} implies that
\begin{equation*}
y^k=J_{\delta S}(w^k) \to J_{\delta S}(w^\star)=:y^\star \in \zer(Q+S). 
\end{equation*}
By the convergence of $w^k$ and $y^k$ we see that
\begin{equation*}
z^k=\frac{1}{\delta}(w^k-y^k)\to \frac{1}{\delta}(w^\star-y^\star)=:z^\star.
\end{equation*}
Finally, \Cref{f:aDR_comono}\Cref{f:aDR_comonoIII} implies that $y^k-v^k\to 0$, thus, $v^k\to y^\star$ and
\begin{equation*}
Mx^k\to z^\star.
\end{equation*}
We prove that there exists $x^\star\in\R^n$ such that $Mx^\star=z^\star$ and that $(x^\star,z^\star,y^\star)$ is a critical point. Indeed, we note that
\begin{equation}\label{eq:graph_inc}
\seq{y^k,z^k}{k}\subseteq \gra S \quad\text{and}\quad \seq{v^k,-Mx^{k+1}}{k}\subseteq \gra Q.
\end{equation}
By taking limits in \eqref{eq:graph_inc} and by recalling \Cref{l:graph}, we conclude that $(y^\star,z^\star)\in \gra S$ and $(y^\star,-z^\star)\in \gra Q$, equivalently,
$$ y^\star\in\subd g(z^\star) \quad\text{and}\quad -M^Ty^\star\in\partial f(x^\star),$$
for some $x^\star$ such that $z^\star=Mx^\star$. Finally, since we now have a critical point $(x^\star,z^\star,y^\star)$, we invoke \cref{lem:critpoint} which concludes the proof.%
\end{proof}

\begin{remark}[on the strongly-weakly convex assumption]\label{rm:wsc1}
In the strongly-weakly convex settings, certain assumptions are usually imposed on the parameters $(\alpha,\beta)$. In our adaptive approach, \cref{as:conv_param} requires
\begin{equation*}
\alpha+\beta\|M\|^2\geq 0.
\end{equation*}
This is an improvement of \cite[Assumption~3.1(i)]{wADMM} for the ADMM, which can be reformulated in the form
\begin{equation}\label{eq:ass_strong}
\alpha+\beta\|M\|^2>0.
\end{equation}
In a different algorithmic approach, the convergence of the \emph{Primal-Dual Hybrid Gradient method} for the strongly-weakly convex setting was established in \cite{MSMC15} by assuming \eqref{eq:ass_strong} (see~\cite[Theorem 2.3]{MSMC15}).
Summing up, we see that our analysis is applicable in the setting of the classical ADMM for convex functions (i.e., $\alpha=\beta=0$) while \cite{MSMC15,wADMM} are not. This observation aligns with the analysis of the adaptive DR algorithm in \cite{BDP19,DPadapt} (see \cite[Remark~5.5]{DPadapt}).
\end{remark}

\begin{remark}[parameters in the classical ADMM]\label{rm:wsc2}
We consider the classical ADMM in the case of strong convexity $\alpha+\beta\|M\|^2>0$. We then set $\gamma=\delta\in\R_{++}$ for $\delta$ such that $\delta+2\beta>0$ and
\begin{align*}
\delta+2\beta-\sqrt{2\left(\frac{\alpha}{\|M\|^2}+\beta\right)(\delta+2\beta)}<\delta<\delta+2\beta+\sqrt{2\left(\frac{\alpha}{\|M\|^2}+\beta\right)(\delta+2\beta)}.
\end{align*}
which is equivalent to
\begin{align*}
\delta > \max\left\{-2\beta,\frac{-2\alpha\beta}{\alpha+\beta\|M\|^2}\right\}=\frac{-2\alpha\beta}{\alpha+\beta\|M\|^2}, \quad\text{if } \beta\leq 0.
\end{align*}
This range of parameters improves the results in \cite{wADMM} which require (see \cite[Assumption~3.1(ii)]{wADMM})
\begin{align*}
\delta > -2\beta+\frac{8\beta^2\|MM^T\|}{\alpha+\beta\|M\|^2}
=\frac{-2\alpha\beta}{\alpha+\beta\|M\|^2}+\frac{6\beta^2\|M\|^2}{\alpha+\beta\|M\|^2}.
\end{align*}
\end{remark}

\begin{remark}[Self-duality of the aDR]
In the proof of \Cref{t:cvg_aadmm} we have shown that the aADMM algorithm is, in fact, a dual aDR iteration. We now discuss the case where $M=\Id$ in \eqref{eq:P}. In this case, it is known that the classical DR is self-dual in the sense of  \cite[Proposition~3.43]{EckThesis}, see also, \cite[Corollary~4.3]{BBHM12}. It is not difficult to show that the adaptive DR is also self-dual. Indeed, in the case where $M=\Id$, iteration \eqref{eq:dual_aDR} is the aDR applied to the operators
\begin{equation*}
S=(\subd g)^{-1} \quad\text{and}\quad Q=(-\Id)\circ(\partial f)^{-1}\circ (-\Id).
\end{equation*}
Therefore, one easily checks that
\begin{subequations}\label{e:proxQS}
\begin{align}
J_{\delta S}&=\Id-\delta\left(\prox_{\tfrac{1}{\delta}g}\right)\circ\left(\frac{1}{\delta}\Id\right),\label{e:proxS}\\
J_{\gamma Q}&=\Id+\gamma\left(\prox_{\tfrac{1}{\gamma}f}\right)\circ\left(-\frac{1}{\gamma}\Id\right).\label{e:proxQ}
\end{align}
\end{subequations}
By substituting \eqref{e:proxQS} into \eqref{eq:dual_aDR} and taking into account \eqref{e:lm_mu} and the change of variable $t^{k}:=\frac{1}{\delta}w^k$, we obtain the scheme
\begin{equation*}
t^{k+1}=(1-\kappa)t^k + \kappa R_2(R_1 (t^{k})),
\end{equation*}
where 
\begin{align*}
R_1:&=(1-\lambda)\Id + \lambda\prox_{\tfrac{1}{\delta}g},\\
R_2:&=(1-\mu)\Id + \mu\prox_{\tfrac{1}{\gamma}f}.
\end{align*}
We see that the aADMM with $M=\Id$ is  the aDR in \cite[Theorem~5.4]{DPadapt} applied to $g$ and $f$ with parameters $(\delta^{-1},\gamma^{-1},\lambda,\mu)\in\R^4_{++}$ and $\kappa=\frac{\lambda-1}{\lambda}\in{]0,1[}$.
\end{remark}

In \cref{t:cvg_aadmm} we see that the sequence $\seq{Mx^k}{k}$ converges, however there is no indication as to whether $\seq{x^k}{k}$ converges or not. Similarly to the classical case (see, for instance, \cite[Proposition 2.2]{ACL16} and \cite[Proposition 2]{Ts91}), this can be remedied if we assume that $M^TM$ is invertible or $f$ is strongly convex ($\alpha>0$).

\begin{theorem}[convergence of the aADMM under stronger assumptions]
\label{t:cvg_aadmm2}
Suppose that \Cref{as:conv_param}, \Cref{as:critpoint} and one of the following conditions hold:
\begin{enumerate}
\item  $f$ is strongly convex (i.e., $\alpha>0$),
\item  $M^TM$ is invertible.
\end{enumerate}%
Let $\delta>\max\{0,-2\beta\}$ and let $\gamma>0$ {satisfy}~\eqref{eq:ass_param}. 
Set $(x^0,z^0,y^0)\in\R^n\times\R^m\times\R^m$ and let $\seq{x^k,z^k,y^k}{k}$ be generated by the aADMM~\eqref{alg}. Then
\begin{equation*}
x^k\to x^\star,\quad z^k\to z^\star \quad\text{and}\quad y^k\to y^\star,
\end{equation*}
where $(x^\star,z^\star,y^\star)$ is a critical point of $L_0(x,z,y)$. Consequently, $(x^\star,z^\star)$ is a solution of \eqref{eq:P}.
\end{theorem}
\begin{proof} By employing \Cref{l:QCequiv}\ref{l:QCequiv-ivb} and \Cref{l:QCequiv}\ref{l:QCequiv-v}, respectively, we see that the conditions in \ref{l:xk_unique_f} and \ref{l:xk_unique_M} imply \Cref{as:QC}. Consequently, we employ \cref{t:cvg_aadmm} in order to obtain a critical point $(x^\star,z^\star,y^\star)$ of $L_0(x,z,y)$ where
	$$
Mx^k\to Mx^\star,\quad z^k\to z^\star \quad\text{and}\quad y^k\to y^\star.	
$$
Set $v^k:=y^k+\gamma(Mx^{k+1}-z^k)$, $k=0,1,\ldots$. Then $v^k\to y^\star$. If $f$ is strongly convex, then, by employing~\Cref{l:xk_unique}\ref{l:xk_unique_f}, we conclude that $(\partial f)^{-1}$ is Lipschitz continuous and
$$x^{k+1}=(\partial f)^{-1}(-M^Tv^k) \to (\partial f)^{-1}(-M^Ty^\star)=x^\star.$$
Suppose now that $M^TM$ is invertible. Then, by employing \Cref{l:xk_unique}\ref{l:xk_unique_M}, we conclude that
$$ x^{k+1}=(M^TM)^{-1}M^T\left(\tfrac{1}{\gamma}(v^k-y^k)+z^k\right) \to (M^TM)^{-1}M^Tz^\star=x^\star,$$
which completes the proof.%
\end{proof}

\subsection{Summary: Existence and Convergence}

For the sake of convenience and accessibility, we recollect and unify all of the conditions from our discussion regarding the existence of a solution of \eqref{eq:P} and the convergence of the aADMM and the ADMM.

\begin{corollary}[existence and convergence of aADMM]\label{cor:aADMM}
Let $M\in\R^{m\times n}$ be a nonzero matrix, let $f:\R^n\to\left]-\infty,+\infty\right]$ be proper, lower semicontinuous and $\alpha$-convex, and let $g:\R^m\to\left]-\infty,+\infty\right]$ be proper, lower semicontinuous and $\beta$-convex where $\alpha, \beta\in\R$ are parameters such that 
\begin{equation*}\label{eq:corAs1}
\alpha\geq 0\quad\text{and}\quad\alpha+\beta\|M\|^2\geq 0.
\end{equation*}
Suppose that one of the following conditions holds:
\begin{enumerate}[label={\rm(A.\arabic*)},leftmargin=0.6in]
\item\label{A1} the Lagrangian $L_0$ has a critical point,
\item\label{A2} the Lagrangian $L_0$ has a saddle point,
\item\label{A3} problem \eqref{eq:P} has an optimal solution and $0\in \ri(\dom g - M(\dom f))$;
\end{enumerate}
and that one of the following conditions holds:
\begin{enumerate}[label={\rm(B.\arabic*)},leftmargin=0.6in]
\item\label{B1} $0\in\ri(\dom f^* - \ran M^T)$,
\item\label{B2} $\ri(\ran \partial f)\cap \ran M^T \neq\varnothing$,
\item\label{B3} $(\rec f)(x)>0$ for all $x\in \ker M\setminus\{x\in\R^n : -(\rec f)(-x)=(\rec f)(x)=0\}$,
\item\label{B4} $f$ is coercive (in particular, supercoercive),
\item\label{B5} $\alpha > 0$ (i.e., $f$ is strongly convex),
\item\label{B6} $M^TM$ is invertible.
\end{enumerate}
Let $\delta>\max\{0,-2\beta\}$ and set
\begin{align*}
&\gamma=\delta+2\beta, & \text{if } \alpha+\beta\|M\|^2=0,\\
&\gamma\in\left]\max\{0,\delta+2\beta - \Delta_{\delta}\}, \delta+2\beta + \Delta_{\delta} \right[, & \text{if } \alpha+\beta\|M\|^2>0;
\end{align*}
where 
$$\Delta_{\delta}:=\frac{1}{\|M\|}\sqrt{2\left({\alpha}+\beta{\|M\|^2}\right)(\delta+2\beta)}.$$ 
Set $(x^0,z^0,y^0)\in\R^n\times\R^m\times\R^m$ and let $\seq{x^k,z^k,y^k}{k}$ be generated by the aADMM \eqref{alg}. Then
\begin{equation*}
Mx^k\to Mx^\star,\quad z^k\to z^\star \quad\text{and}\quad y^k\to y^\star,
\end{equation*}
where $(x^\star,z^\star,y^\star)$ is a critical point of $L_0(x,z,y)$. Consequently, $(x^\star,z^\star)$ is a solution of \eqref{eq:P}. If, in particular,  \ref{B5} or  \ref{B6} holds, then $x^k\to x^\star$.
\end{corollary}
\begin{proof}
Clearly, \Cref{as:conv_param} holds. By \Cref{l:qc-cp} and \Cref{l:cp-sp}, each one of the conditions \ref{A1}--\ref{A3} implies \Cref{as:critpoint}. Finally, by \Cref{l:QCequiv}, each one of the conditions \ref{B1}--\ref{B6} implies \Cref{as:QC}. We conclude the proof by invoking \Cref{t:cvg_aadmm} and \Cref{t:cvg_aadmm2}.%
\end{proof}

As we point out in \Cref{rm:wsc1}, our assumptions on \eqref{eq:P} extend the framework of the classical ADMM for two convex functions.

\begin{corollary}[existence and convergence of ADMM]\label{cor:ADMM}
Let $M\in\R^{m\times n}$ be a nonzero matrix and let $f:\R^n\to\left]-\infty,+\infty\right]$ and $g:\R^m\to\left]-\infty,+\infty\right]$ be proper, lower semicontinuous and convex. Suppose that one of the following conditions holds:
\begin{enumerate}[label={\rm(A.\arabic*)},leftmargin=0.6in]
\item\label{A1c} the Lagrangian $L_0$ has a saddle point (equivalently, a critical point),
\item\label{A2c}  problem \eqref{eq:P} has an optimal solution and $0\in \ri(\dom g - M(\dom f))$;
\end{enumerate}
and that one of the following conditions holds:
\begin{enumerate}[label=\text{\rm(B.\arabic*)},leftmargin=0.6in]
\item\label{B1c} $0\in\ri(\dom f^* - \ran M^T)$,
\item\label{B2c} $\ri(\ran \partial f)\cap \ran M^T \neq\varnothing$,
\item\label{B3c} $(\rec f)(x)>0$ for all $x\in \ker M\setminus\{x\in\R^n : -(\rec f)(-x)=(\rec f)(x)=0\}$,
\item\label{B4c} $f$ is coercive (in particular, supercoercive),
\item\label{B5c} $f$ is strongly convex,
\item\label{B6c} $M^TM$ is invertible.
\end{enumerate}
Let $\gamma>0$. Set $(x^0,z^0,y^0)\in\R^n\times\R^m\times\R^m$ and let $\seq{x^k,z^k,y^k}{k}$ be generated by the ADMM \eqref{alg0}. Then
\begin{equation*}
Mx^k\to Mx^\star,\quad z^k\to z^\star \quad\text{and}\quad y^k\to y^\star,
\end{equation*}
where $(x^\star,z^\star,y^\star)$ is a critical point of $L_0(x,z,y)$. Consequently, $(x^\star,z^\star)$ is a solution of \eqref{eq:P}. If, in particular, \ref{B5c} or \ref{B6c} holds, then $x^k\to x^\star$.
\end{corollary}
\begin{proof}
The proof follows from \Cref{cor:aADMM} when we set $\alpha=\beta=0$.%
\end{proof}

\subsection{Optimality Conditions and Stopping Criteria}
\label{ss:stop}

By following \cite[\S~3.3]{BPCPE11}, we discuss the primal-dual residual stopping criteria for the aADMM which we will employ in \cref{sec:exp} for our numerical experiments.

By \Cref{d:crit},  $(x^\star,z^\star,y^\star)$ is a critical point if
\begin{subequations}\label{e:stop1}
\begin{align}
&0\in\partial f(x^\star)+M^Ty^\star,\label{e:stop1a}\\
&0\in\subd g(z^\star)-y^\star,\label{e:stop1b}\\
&Mx^\star-z^\star=0.\label{e:stop1c}
\end{align}
\end{subequations}
Clearly, the sequence $r^{k}:=Mx^{k}-z^{k}$ can be viewed as residual for \eqref{e:stop1c}.
By optimality at each step in \eqref{alg} we obtain
\begin{subequations}
\begin{align*}
0&\in\partial f(x^{k+1})+M^Ty^k+\gamma M^T(Mx^{k+1}-z^k),\\
0&\in\subd g(z^{k+1})-y^k-\delta(Mx^{k+1}-z^{k+1}).
\end{align*}
\end{subequations}
We note that $y^{k+1}=y^k+\delta(Mx^{k+1}-z^{k+1})$. Consequently,
\begin{subequations}
\begin{align*}
0&\in\partial f(x^{k+1})+M^Ty^{k+1}-\delta M^T(Mx^{k+1}-z^{k+1})+\gamma M^T(Mx^{k+1}-z^k),\\
0&\in\subd g(z^{k+1})-y^{k+1}.
\end{align*}
\end{subequations}
We see that \eqref{e:stop1b} holds at each step whereas for \eqref{e:stop1a} we will monitor the dual residual
\begin{equation*}\label{eq:dual_res}
\begin{aligned}
s^{k+1}:= &-\delta M^T(Mx^{k+1}-z^{k+1})+\gamma M^T(Mx^{k+1}-z^k)\\
= & M^T(\gamma z^{k}-\delta z^{k+1}-(\gamma-\delta)Mx^{k+1}).
\end{aligned}
\end{equation*}
We will therefore employ the primal-dual residual stopping criteria
\begin{subequations}\label{e:stop-primaldual}
\begin{align}
\|r^k\|:=\|Mx^{k}-z^{k}\| & \leq  \epsilon_{\text{primal}}, \\
\|s^k\|:=\|M^T(\gamma z^{k-1}-\delta z^k-(\gamma-\delta)Mx^k)\|& \leq \epsilon_{\text{dual}},
\end{align}
\end{subequations}
for some fixed $\epsilon_{\text{primal}}>0$ and $ \epsilon_{\text{dual}}>0$. We note that the case where $\delta=\gamma$ in \eqref{e:stop-primaldual} coincides with the primal-dual stopping criteria employed for the classical ADMM (see~\cite[\S~3.3]{BPCPE11}).
Finally, $\epsilon_{\text{primal}}$ and $\epsilon_{\text{dual}}$ may be chosen using absolute and relative tolerances, for example,
\begin{align*}
\epsilon_{\text{primal}} & = \sqrt{m}~\epsilon_{\text{abs}}+\epsilon_{\text{rel}}\max\{\|Mx^k\|,\|z^k\|\},\\
\epsilon_{\text{dual}} & = \sqrt{n}~\epsilon_{\text{abs}}+\epsilon_{\text{rel}}\|M^Tz^k\|,
\end{align*}
where $m$ and $n$ are the dimensions of the matrix $M$.

\begin{remark}[Stopping criteria from the aDR perspective]
As a counterpart, we examine the stopping criteria from the perspective of the aDR algorithm. By employing the notations in the proof of \cref{t:cvg_aadmm}, we consider
\begin{equation*}
w^k:=y^k+\delta z^k
\quad\implies\quad y^k=J_{\delta S}(w^k).
\end{equation*}
Since $w^k\to w^\star\in \Fix T_{\text{aDR}}$, it is reasonable to employ the Cauchy-type stopping criteria
\begin{equation}\label{e:cauchy_adr}
\|w^{k+1}-w^k\|\leq \epsilon_{\text{aDR}}.
\end{equation}
We observe that
\begin{align*}
r^k&=\delta^{-1}(y^k-y^{k-1}),\\
s^k&=M^T((\delta-\gamma)Mx^k -\delta z^k+\gamma z^k) + M^T\gamma(z^{k-1}-z^k)\\
&=M^T(\delta-\gamma)r^k+M^T\gamma(z^{k-1}-z^{k}).
\end{align*}
Consequently, \eqref{e:stop-primaldual} is equivalent to $\|y^{k+1}-y^k\|$ and $\|z^{k+1}-z^k\|$ being small. By the triangle inequality, 
\begin{equation}\label{eq:wk1}
\|w^{k+1}-w^k\|\leq \|y^{k+1}-y^k\|+\delta\|z^{k+1}-z^k\|.
\end{equation}
On the other hand, since $z^k\in S(y^k)$ and since $S$ is maximally $\beta$-comonotone, we obtain
\begin{equation}\label{eq:wk2}
\begin{aligned}
\|w^{k+1}-w^k\|^2&=\|(y^{k+1}-y^k) + \delta(z^{k+1}-z^k)\|^2\\
&=\|y^{k+1}-y^k\|^2+\delta^2\|z^{k+1}-z^k\|^2
+2\delta\langle y^{k+1}-y^k, z^{k+1}-z^k\rangle\\
&\geq \|y^{k+1}-y^k\|^2+\delta(\delta+2\beta)\|z^{k+1}-z^k\|^2.
\end{aligned}
\end{equation}
In view of \eqref{eq:wk1} and \eqref{eq:wk2}, we see that the Cauchy stopping criteria \eqref{e:cauchy_adr} with appropriately chosen $\epsilon_{\text{aDR}}$ is equivalent to~\eqref{e:stop-primaldual}, which justifies the use of primal-dual residual stopping criteria.
\end{remark}

\section{Numerical Experiments}\label{sec:exp}

We now examine the applicability and efficiency of our aADMM with numerical experiments. To this end we focus on a \emph{total variation} signal denoising problem for which the classical ADMM has been widely used (see, e.g.,~\cite{CDV10,CP11,CW05,WBAW12} and the references therein). All of our codes are in Python~3.7. The datasets generated during and/or analysed during the current study are available from the corresponding author on reasonable request.

Suppose that a discretized observed signal $\widehat \phi\in\R^n$ is the result of
$$\widehat \phi= \phi + \xi,$$
where $\phi\in \R^n$ is the original signal and $\xi\in\R^n$ is a Gaussian noise with 0 mean and variance $\sigma^2$. The objective of the denoising is to obtain an accurate approximation of $\phi$ from $\widehat \phi$. A common approach consists of solving the total variation regularization problem \cite{Rudin} (see, e.g., \cite[\S 6.4.1]{BPCPE11})
\begin{equation}\label{eq:problem_denoise}
\min_{x\in\R^n} \frac{1}{2}\|x-\widehat \phi\|^2 + \omega P(Dx),
\end{equation}
where $\omega>0$ is the regularization parameter, $P:\R^{n-1}\to\R_+$ is a penalty function in order to induce sparsity and $D\in\R^{n\times (n-1)}$ is the total variation matrix defined component-wise by
\begin{equation*}
D_{ij}:=\left\{\begin{array}{ll}
                1, &\text{if } j=i;\\
                -1, &\text{if } j=i+1;\\
                0, &\text{otherwise.}
               \end{array}\right.
\end{equation*}
We observe that problem \eqref{eq:problem_denoise} is of the form \eqref{eq:P} when we let
\begin{equation*}
f:=\frac{1}{2}\|\cdot-\widehat \phi \|^2,\ g=\omega P(\cdot),
\text{~and~}M:=D.
\end{equation*}
Moreover, the particular structure of this problem enables an easy computation of the  ADMM iteration, as we show next.

\paragraph{$x$-update for quadratic functions:}
Since $f$ is a quadratic function, the $x$-update of the ADMM and the aADMM can be obtained via the solution of a system of linear equations. More precisely, steps \eqref{alg0:x} and \eqref{alg:x} are computed by solving the linear system
\begin{equation*}
(\Id+\gamma M^TM)x^{k+1}=M^T(\gamma z^k-y^k)+\widehat\phi.
\end{equation*} 

\paragraph{$z$-update as a proximal step:}
The minimization step with respect to $z$ can be computed via the proximity operator of $g$. Indeed, \eqref{alg:z} reduces to
\begin{equation}\label{eq:z-update prox}
z^{k+1}=\prox_{\frac{1}{\delta}g}(Mx^{k+1}+\delta^{-1}y^k)=\prox_{\frac{\omega}{\delta}P}(Mx^{k+1}+\delta^{-1}y^k),
\end{equation}
whereas for \eqref{alg0:z} one only replaces $\delta$ by $\gamma$ in~\eqref{eq:z-update prox}.

We assume that the penalty function is separable in the sense that there exists $p:\R\to\R_+$ such that
\begin{equation*}
P(z)=\sum_{i=1}^{n-1} p(z_i), \quad\text{for all } z=(z_1,\ldots,z_{n-1})\in\R^{n-1}.
\end{equation*} 
The proximal mapping of $P$, usually referred to as \emph{thresholding} or \emph{shrinkage}, can be computed component-wise. The type of penalty function chosen in \eqref{eq:problem_denoise}, together with the regularization parameter $\omega$, has a significant impact on the quality of the denoised solution. In \Cref{tbl:thresholding} we list three types of penalty functions and their corresponding thresholding mappings (see, e.g., \cite{WCLQ18}).

\begin{table}[ht!]
\centering\small
\def\arraystretch{0.8}
\begin{tabular}{|Sc|Sc|}
\hline
Penalty function & Thresholding mapping \tabularnewline
\hhline{==}

& \emph{Hard} \tabularnewline
 \(\displaystyle p^{(H)}(x)=\begin{cases}
                    0, & \text{if } x=0,\\
                    1, & \text{otherwise}.
                 \end{cases}\)  & $\prox_{\gamma p^{(H)}}(x)=\begin{cases}
                 0, & \text{if } |x|< \sqrt{2\gamma},\\
                 \{0,x\}, & \text{if } |x|= \sqrt{2\gamma}\\
                 x, & \text{if } |x|> \sqrt{2\gamma}
                 \end{cases}$ \tabularnewline\hline

& \emph{Soft} \tabularnewline $p^{(S)}(x)=|x|$  & $\prox_{\gamma p^{(S)}}(x)=\begin{cases}
                                                  0, & \text{if } |x|\leq \gamma,\\
                                                  \sign(x)(|x|-\gamma),& \text{otherwise.}
                                                 \end{cases}$  \tabularnewline\hline
                
& \emph{Firm} \tabularnewline $p^{(F)}_{\zeta}(x):=\begin{cases}
|x|-\tfrac{x^2}{2\zeta}, & \text{if } |x|\leq\zeta,\\
\tfrac{\zeta}{2}, & \text{otherwise.}\end{cases}$ & $\prox_{\gamma p^{(F)}_{\zeta}}(x)=\begin{cases}
0, & \text{if } |x|\leq \gamma,\\
\frac{\sign(x)(|x|-\gamma)\zeta}{\zeta-\gamma}, & \text{if } \gamma<|x|<\zeta,\\
x, & \text{otherwise.}\end{cases}$ \tabularnewline \hline
\end{tabular}
\caption{Three penalty functions and the corresponding thresholding mappings.}\label{tbl:thresholding}
\end{table}

The \emph{hard} penalty is associated with the $l_0$-norm, it is the one to impose sparsity on the solution. Since the hard penalty is nonconvex (in fact, it is not $\alpha$-convex for any choice of $\alpha\in\R$), there is no guarantee for the convergence of the ADMM when applied to this problem. As an alternative, the $l_0$-norm is replaced by the $l_1$-norm, leading to the proximal mapping known as \emph{soft} thresholding. Although convex, this penalty may yield to biased solutions when the variation of the original signal is large. In order to obtain less biased solutions, some weakly convex penalties, such as the \emph{firm} thresholding~\cite{GB97}, associated to a minimax concave penalty, arises as a tradeoff between the hard and the soft thresholdings (see~\Cref{fig:penalties}).   

\begin{figure}[ht!]\centering
\subfigure[Penalty functions]{\includegraphics[width=0.365\textwidth]{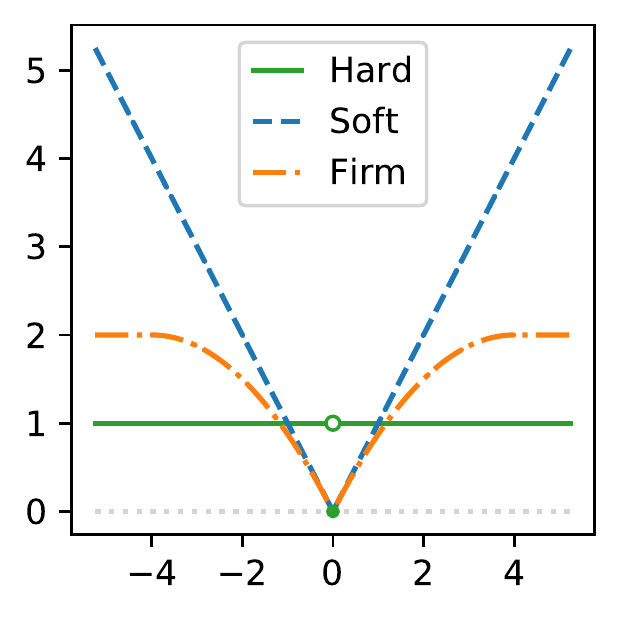}}\hspace{0.05\textwidth}
\subfigure[Thresholding mappings]{\includegraphics[width=0.38\textwidth]{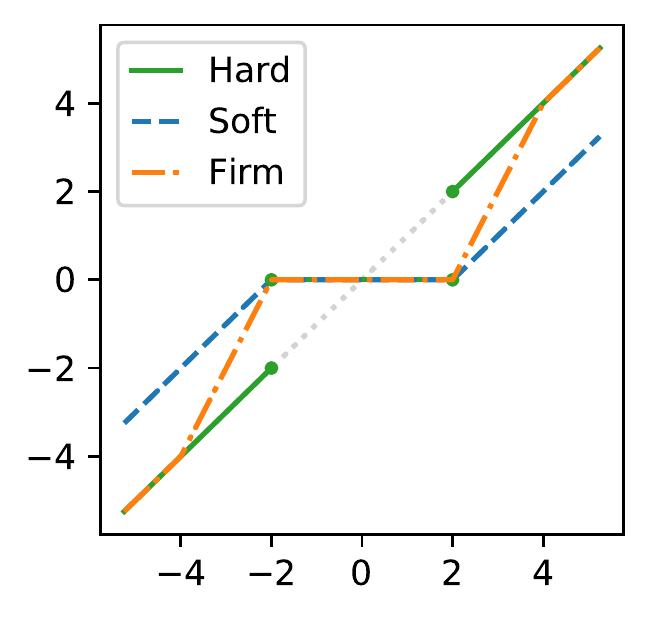}}
\caption{Illustration of the hard, soft and firm penalty (with $\zeta=4$) functions and the corresponding thresholding mappings (with $\gamma=2$).}\label{fig:penalties}
\end{figure}

We observe that the firm penalty function $p^{(F)}_{\zeta}(x)$ is weakly convex with parameter $\frac{-1}{\zeta}$. Indeed, one verifies that the function $p^{(F)}_{\zeta}(x)+\frac{x^2}{2\zeta}$ is convex either by direct computation or via a more general argument such as \cite[Theorem~5.4]{BLP16}.

Since $f$ is $1$-strongly convex, \Cref{as:conv_param} holds whenever $\zeta\geq 4\omega$ since
\begin{equation*}\label{ass:zeta}
1-\|D\|^2\frac{\omega}{\zeta}\geq 0
\quad \iff\quad \zeta\geq\|D\|^2\omega 
\end{equation*}
and the fact that $\|D\|\leq 2$.

\subsection*{Experiment 1: Convex versus Weakly Convex Penalties}

In our first experiment we aim to show that using a weakly convex penalty function may produce more accurate solutions than the $l_1$-norm in certain circumstances. To this end we employ generated\footnote{Signals were generated with the Python package \texttt{PyWavelets} \href{https://pywavelets.readthedocs.io/}{https://pywavelets.readthedocs.io/}} \emph{block signals} (piecewise constant signals) to which we add a gaussian noise with standard deviation $\sigma=0.5$. Given a generated noisy signal, for $50$ values of the penalty parameter $\omega$ equally distributed in the interval ${[0.1,5]}$, we compute two solutions of problem~\eqref{eq:problem_denoise}, one with respect to the soft thresholding (convex) and another with respect to the firm one (weakly convex) with $\zeta=4\omega$. In order to measure the quality of the reconstruction of a denoised signal $x=(x_1,\ldots,x_n)\in\R^n$ with respect to the original signal $\phi=(\phi_1,\ldots,\phi_n)\in\R^n$, we employ the \emph{mean absolute error (MAE)}, defined by
\begin{equation*}
\operatorname{MAE}(x,\phi)=\frac{1}{n}\sum_{i=1}^n |x_i-\phi_i|.
\end{equation*}
The results for a signal of size $n=256$ are shown in \Cref{fig:signal_penalty}. We see that the weakly convex penalty (firm)  achieves more accurate solutions. Indeed, the MAE of the solution produced by the firm thresholding is always smaller than that of the soft thresholding, except for small values of $\omega$ for which the quality of both solutions is deficient. In order to visualize the performance of each penalty function we plot in \Cref{fig:signals} the denoised solutions for $\omega=2$, as well as the original and the noisy signals. We see that the firm thresholding produces less biased solutions at the break points of the signal, in particular, at the break points corresponding to relatively short pieces, where the true variation of the original signal is large.

\begin{figure}[ht!]\centering
\subfigure[MAE with respect to $\omega$\label{fig:signal_penalty}]{\includegraphics[width=0.34\textwidth]{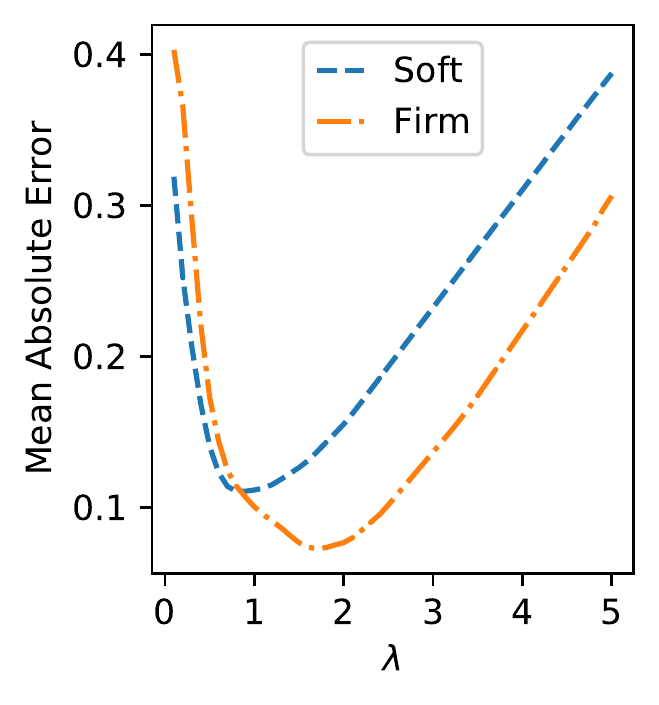}}\hspace{7mm}
\subfigure[Original, noised and denoised signals for $\omega=2$\label{fig:signals}]{\includegraphics[width=0.59\textwidth]{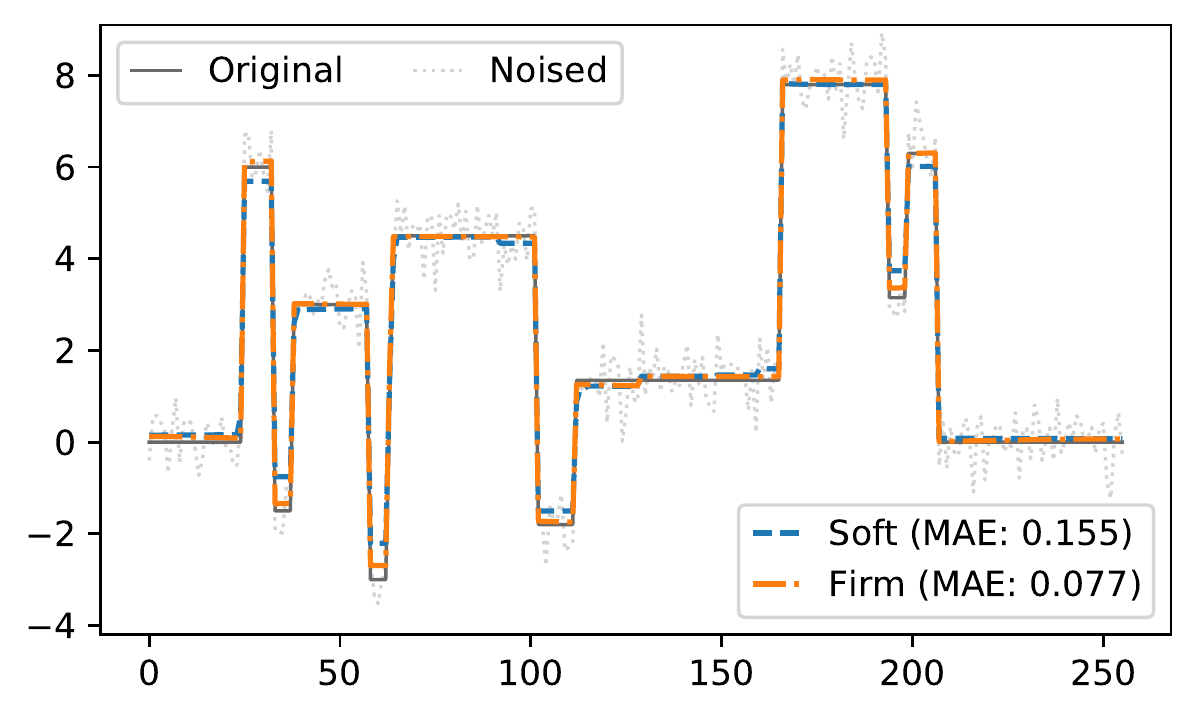}}
\caption{Denoising of a block signal of size $n=256$ with soft and firm thresholdings. (a) Mean absolute error of the solutions for each penalty type with respect to the regularization parameter $\omega$. (b) Plot of the original, noised and the two denoised signals for $\omega=2$.}\label{fig:exp1}
\end{figure}

We repeated the experiment for different signal size $n$ and different variance of the noise $\sigma$. The results are similar: the firm thresholding produces solutions with less error for all large enough values of the penalty parameter $\omega$. Due to the similarity we do not include these in this presentation.

For denoising with the soft thresholding we used the classical ADMM, whereas for the firm denoising we employed the aADMM. In this experiment we focused on the quality of the solutions produced by each penalty function in order to motivate the application of weakly convex functions, rather than the algorithm performance. In our second experiment we  do focus on the performance of the algorithm.

\subsection*{Experiment 2: ADMM versus aADMM}

We now examine the efficiency of the  aADMM for signal denoising with weakly convex penalty. We focus on solving the total variation regularization problem~\eqref{eq:problem_denoise} with firm thresholding. We consider generated block signals of different sizes $n$, noised with a Gaussian error of standard deviation $\sigma$. For every size, we set the parameters
\begin{equation}\label{numexp_parameters}
\sigma:=0.5,\quad \omega:=2\quad\text{and}\quad \zeta:=4\omega=8. 
\end{equation}
As noted in \Cref{rem:reformulation}, we can apply the classical ADMM to problem~\eqref{eq:problem_denoise} with a weakly convex penalty $P^{(F)}_{\zeta}$ via a  convex reformulation. Thus, we compare the performance between the ADMM and the aADMM. For each $n\in\{1000,2000,\ldots,10000\}$, we generated $10$ noisy signals randomly. Then, for each of these noisy signals, we run the ADMM with 10 random starting points, for each $\gamma\in\{0.2,0.4,\ldots,7.0\}$. At each instance, the aADDM was also launched for the same value of $\gamma$, while the parameter $\delta$ in the $z$-update step was set to
\begin{equation}\label{numexp_delta}
\delta:=\gamma-2\beta=\gamma+\frac{2\omega}{\zeta}=\gamma+\frac{1}{2}
\end{equation}
so that the conditions for convergence in \eqref{eq:ass_param} hold. For both algorithms we employ the primal-dual residuals stopping criteria from \cref{ss:stop} where we set $\epsilon_{\text{abs}}=\epsilon_{\text{rel}}:=10^{-4}$.

The results of the experiments are shown in \Cref{fig:exp2}: We plot the median ratio between the number of iterations required by aADMM and ADMM with respect to $\gamma$, for each size. A ratio less than one indicates that the aADMM converged faster. While, on average, both algorithms behave similarly for large values of $\gamma$, the superiority of aADMM for small values of $\gamma$, where the ratio is always smaller than $1$, is evident in this experiment.

\begin{figure}[ht!]\centering
\vspace{2mm}
\includegraphics[width=0.7\textwidth]{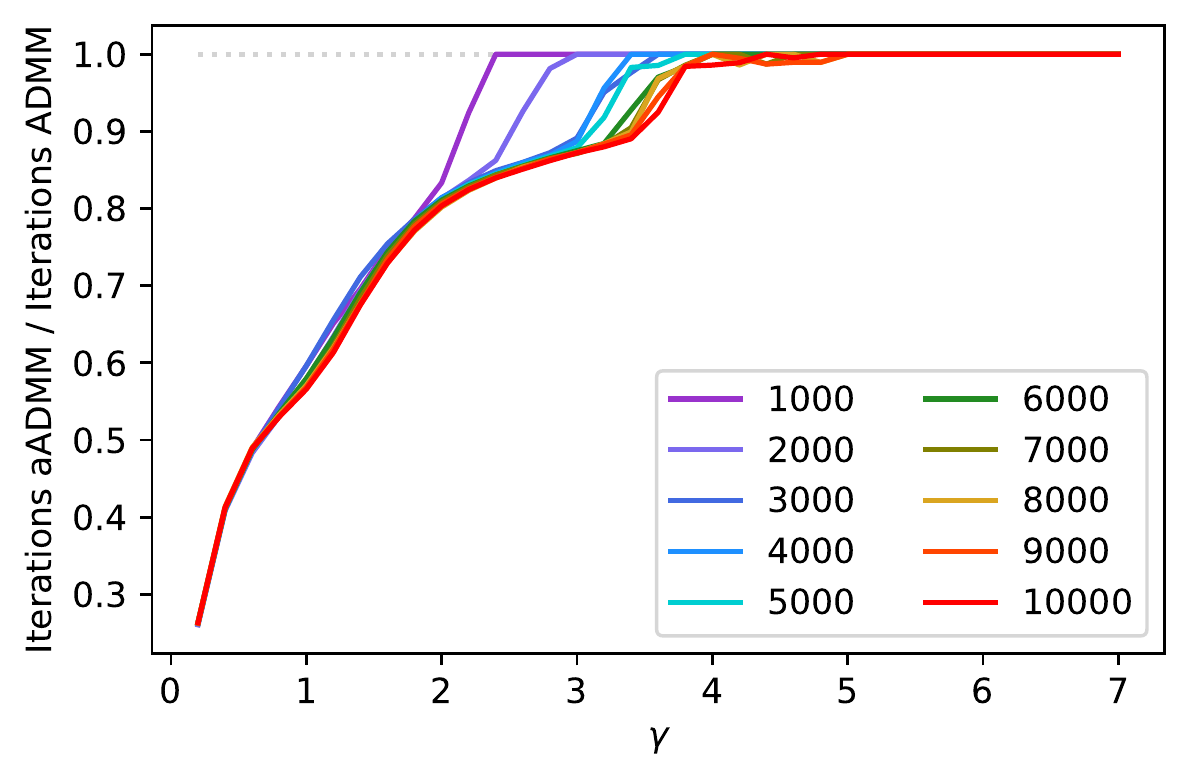}
\caption{Median of 100 ratios between the median (among 50 instances) of iterations required to converge by aADMM and classical ADMM, with respect to the penalty parameter $\gamma$, for denoising signals of sizes $1000,2000,\ldots,10000$ with firm thresholding.}\label{fig:exp2}
\end{figure}

In Figure~\ref{fig:ratio_percent} we plot the percentiles $0,5,\ldots,95$ and $100$ of the ratios between the numbers of iterations with respect to $\gamma$ (for all sizes). This plot exhibits a small variance of the ratio, which demonstrates that the medians in \Cref{fig:exp2} are suitable representatives. Indeed, the curve for 95\%-percentile still lies close to 1, which indicates that the aADMM is at least comparable with (or even better than) the ADMM within 95\% of the times. In fact, the curve for the 70\%-percentile lies entirely below 1. This implies that for all tested values of $\gamma$ and all instances, the aADMM was at least as fast as the ADMM 70\% of the time.

\begin{figure}[ht!]\centering
\includegraphics[width=0.83\textwidth]{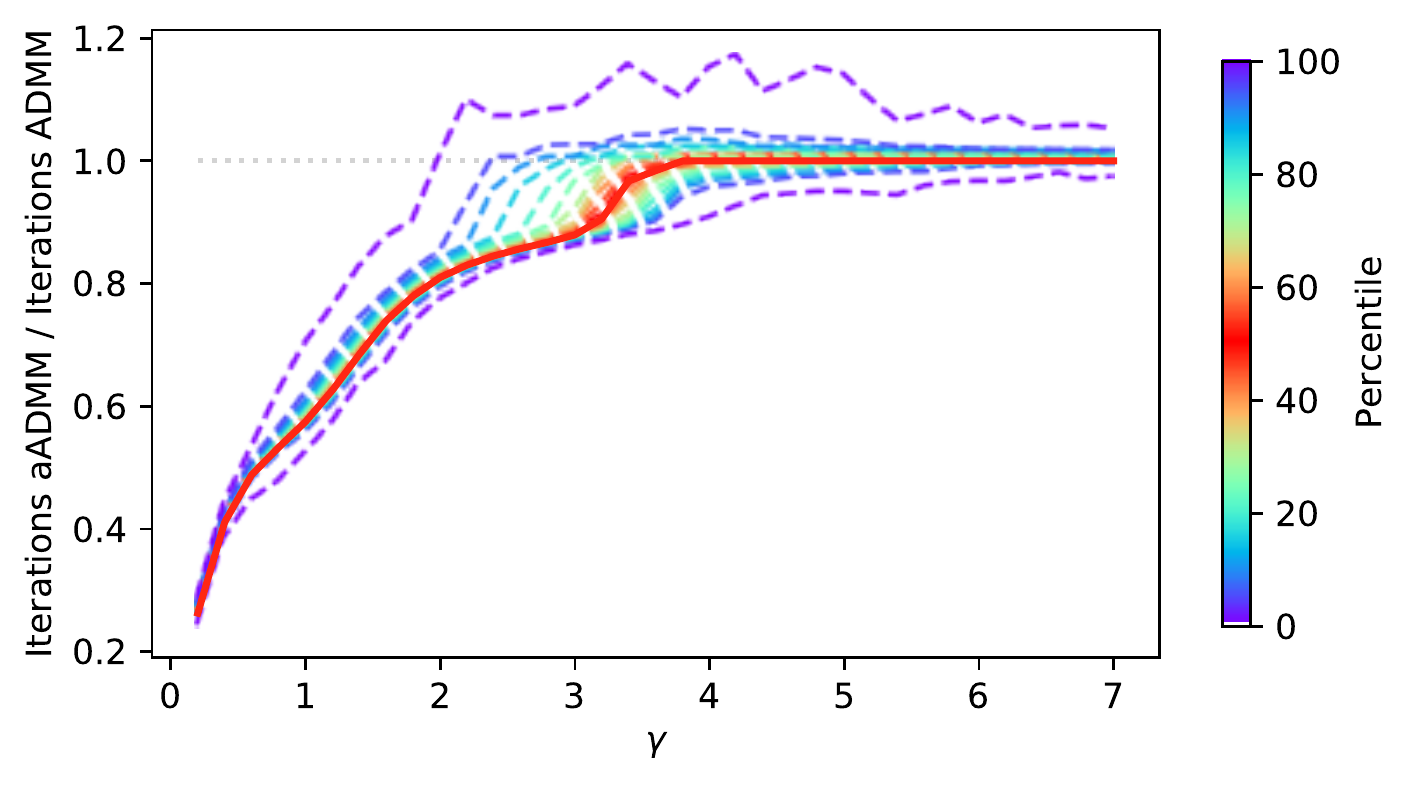}
\caption{Percentile of all 1000 ratios between the numbers of iterations of aADMM and classical ADMM, with respect to the penalty parameter $\gamma$ for denoising signals with firm thresholding.}\label{fig:ratio_percent}
\end{figure}

Finally, in order to better analyze the results, we further examine the outputs for $n=1000$, $n=5000$ and $n=10000$ in \Cref{fig:exp2b}. Instead of the ratios between both algorithms, we plot the median of the number iterations required by each of the algorithms separately. On the one hand, we confirm our previous conclusions: for every fixed value of $\gamma$, the aADMM is as rapid as the classical ADMM (and much faster for small $\gamma$). On the other hand, if one was able to predict the optimal choice for $\gamma$, then both algorithms perform similarly. However, the optimal $\gamma$ is unknown.

\begin{figure}[ht!]\centering
\subfigure[$n=1000$\label{fig:iter1}]{\includegraphics[width=0.325\textwidth]{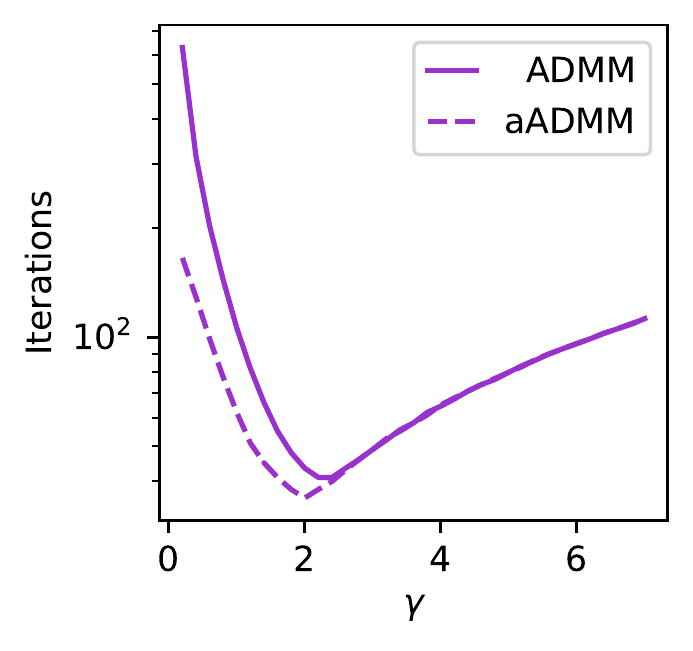}}\hfill
\subfigure[$n=5000$\label{fig:iter4}]{\includegraphics[width=0.325\textwidth]{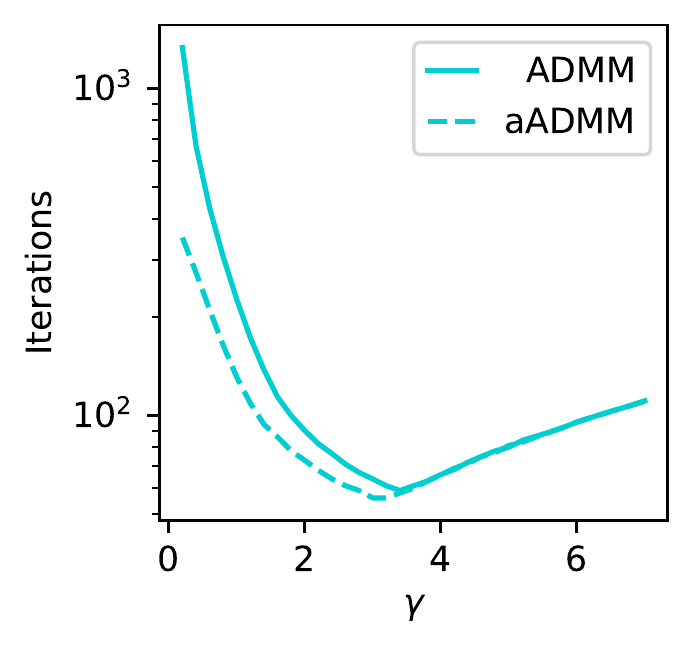}}\hfill
\subfigure[$n=10000$\label{fig:iter9}]{\includegraphics[width=0.325\textwidth]{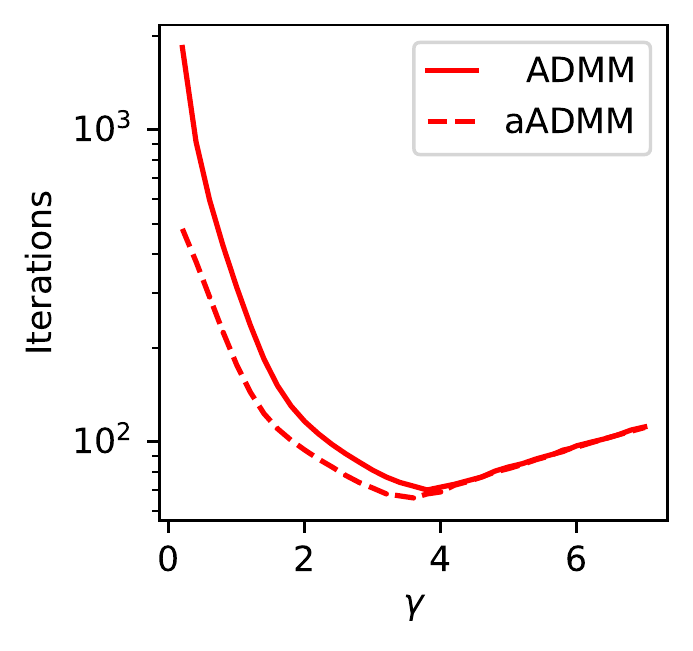}}
\caption{Median (among 100 instances) of iterations required to converge by aADMM and ADMM, with respect to the penalty parameter $\gamma$, for denoising signals of sizes $1000$, $5000$ and $10000$ with firm thresholding.}\label{fig:exp2b}
\end{figure}

\subsection*{Experiment 3: Comparison between ADMM, aADMM and PDHGM}

We conclude our experiments by incorporating the \emph{primal-dual hybrid gradient method}~\cite{MSMC15}, denoted PDHGM in short, into our comparison. For this experiment we used the parameter setting in \eqref{numexp_parameters} but with $\zeta:=9$ so that the restriction in \eqref{eq:ass_strong} holds. This is required for the convergence of PDHGM in the strongly-weakly convex setting.

For a generated noised block signal of length $n=1000$, we run ADMM, aADMM and PDHGM. For ADMM and aADMM, the parameter $\gamma$ was optimally chosen according to the results obtained in \Cref{fig:exp2b}, while $\delta$ in the aADMM was computed as in \eqref{numexp_delta}. The parameters of the PDHGM did not seem to have a big effect on the convergence rate in this experiment.  Nevertheless, we roughly tuned them for best performance. The same experiment was repeated for a larger signal of length $5000$. The results for both signals are shown in \Cref{fig:PDHGM}. We observe superiority of the ADMM/aADMM compared to the PDHGM in these particular instances, with a slight advantage for the aADMM over the classical ADMM, which is in agreement with our previous results.
 
\begin{figure}[ht!]\centering
\subfigure[n=1000]{\includegraphics[width=0.45\textwidth]{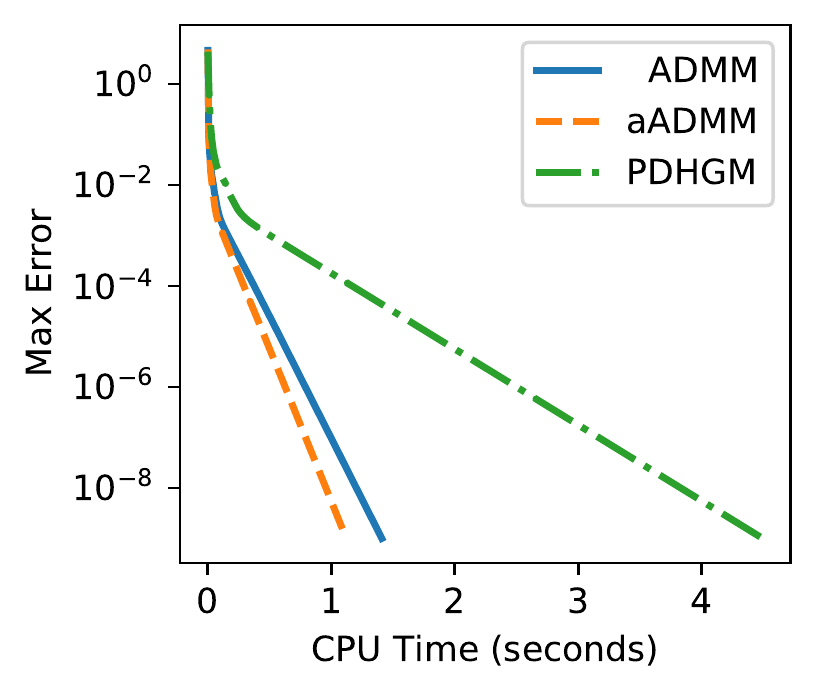}}\hspace{0.05\textwidth}
\subfigure[n=5000]{\includegraphics[width=0.45\textwidth]{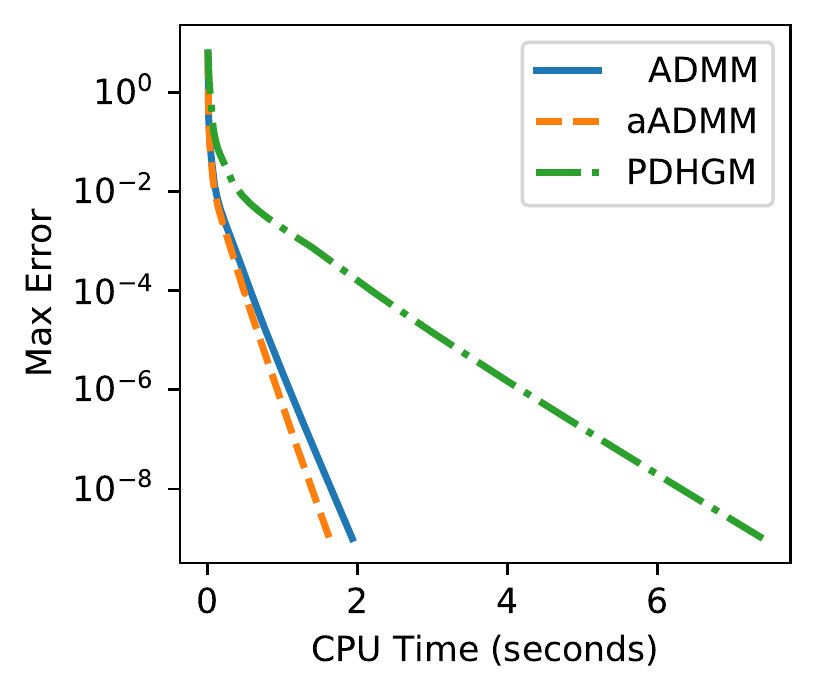}}
\vspace{-0.25cm}
\caption{Comparison between ADMM, aADMM and PDHGM for denoising two signals of sizes $1000$ and $5000$. For each algorithm we plot $\max\{\Vert Mx^{k+1}-z^{k+1}\Vert,\Vert x^{k+1}-x^{k}\Vert, \Vert z^{k+1}-z^{k}\Vert \}$ with respect CPU time in seconds.}\label{fig:PDHGM}
\end{figure}

\section{Conclusions}\label{sec:con}

We provided an adaptive version of the alternating direction method of multipliers for solving linearly constrained optimization problems where the objective is the sum of a weakly convex function and a strongly convex function. Convergence of our scheme was derived as dual to the convergence of the adaptive Douglas--Rachford splitting algorithm. Consequently, the theory regarding the dual relations between the classical ADMM and the classical DR algorithm have been extended to and established in the strongly-weakly convex framework. In the process we have also relaxed stronger assumptions imposed in other studies where the ADMM was applied in this framework  (see~\Cref{rm:wsc1,rm:wsc2}).

The performance of our scheme was tested in numerical experiments of signal denoising. In our experiments we observed that the aADMM outperforms the classical ADMM in terms of the required number of iterations. However, our numerical experiments are  far from providing a complete computational study. A detailed and comprehensive numerical analysis will be the subject matter of future studies.

{\small
\paragraph{\small Acknowledgements}
Sedi Bartz was partially supported by a Simons Foundation Collaboration Grant for Mathematicians, Grant 854168, and by a UMass Lowell faculty startup grant. Rub\'en Campoy was supported, in part, by a postdoctoral fellowship of UMass Lowell, and by the Ministry of Science, Innovation and Universities of Spain
and the European Regional Development Fund (ERDF) of the European Commission, Grant
PGC2018-097960-B-C22 and by the Generalitat Valenciana (AICO/2021/165). Hung M. Phan was partially supported by Autodesk, Inc. via a gift made to the Department of Mathematical Sciences,
UMass Lowell.
Sedi Bartz and Hung M. Phan were partially supported by a Seed Grant from the Kennedy College of Sciences, UMass Lowell.

}

\end{document}